\documentclass[]{amsart}                     

\usepackage{amsmath}
\usepackage{amssymb,amsfonts}
\usepackage{bbm}
\usepackage{xypic}
\xyoption{all}
\usepackage[colorlinks=true,citecolor=blue]{hyperref}

\newtheorem{theorem}{Theorem}[section]
\newtheorem{definition}[theorem]{Definition}
\newtheorem{proposition}[theorem]{Proposition}

\newtheorem{lemma}[theorem]{Lemma}
\newtheorem{remark}[theorem]{Remark}
\newtheorem{example}[theorem]{Example}

\newcommand{\B}{\mathcal{B}}
\newcommand{\HH}{\mathrm{H}}
\newcommand{\Z}{\mathrm{Z}}

\newcommand{\G}{\mathcal{G}}
\newcommand{\M}{\mathcal{M}}
\newcommand{\Sw}{\mathcal{S}}

\newcommand{\Aut}{\mathrm{Aut}}
\newcommand{\I}{\mathrm{I}}
\newcommand{\AAw}{\mathbbm{A}}

\newcommand{\DD}{\mathbb{D}}

\begin{document}

\title{Structure and classification of monoidal groupoids}

\author{M. Calvo \and  A.M. Cegarra \and  B.A. Heredia}

\thanks{The authors want to express their gratitud to the anonymous referee, whose helpful comments and remarks
improved our exposition.
This work has been supported by `Direcci\'on General de
Investigaci\'on' of Spain, Project: MTM2011-22554, and for the third author also by Ministerio de Educaci\'on, Cultura y Deportes of Spain: FPU grant AP2010-3521.}

\address{\newline
Departamento de \'Algebra, Facultad de Ciencias, Universidad de Granada.
\newline 18071 Granada, Spain \newline
 mariacc88@gmail.com \ acegarra@ugr.es\ baheredia@ugr.es }

\subjclass[2000]{18D10;  20M50 ;  18G50; 18G55}

\keywords{Monoidal category; groupoid; non-abelian cohomology ;
monoid cohomology}
\begin{abstract}
The structure of  monoidal categories in which every arrow is
invertible is analyzed in this paper, where we develop a
3-dimensional Schreier-Grothendieck theory of non-abelian factor
sets for their classification. In particular, we state and prove
precise classification theorems for those monoidal groupoids whose
isotropy groups are all abelian, as well as for their homomorphisms,
by means of Leech's cohomology groups of monoids.
\end{abstract}
\maketitle

\section{Introduction and summary} This paper deals with   {\em monoidal
groupoids}
$\G=(\G,\otimes,\mathrm{I},\boldsymbol{a},\boldsymbol{l},\boldsymbol{r})$,
that is, categories $\G$ in which all arrows are invertible,
enriched with a monoidal structure by a tensor product
$\otimes:\G\times\G\to\G$, a unit object $\I$, and coherent
associativity and unit constraints $\boldsymbol{a}:(X\otimes
Y)\otimes Z\cong X\otimes (Y\otimes Z)$, $\boldsymbol{l}:\I\otimes
X\cong X$, and ${\boldsymbol{r}:X\otimes \I\cong X}$. Our main
objective is to state and prove precise classification theorems for
monoidal groupoids and their homomorphisms. In this classification,
two monoidal groupoids, say  $\G$ and $\G'$, are {\em equivalent}
whenever they are connected by a monoidal equivalence
$F:\G\overset{_\sim\,}\to\G'$, and two monoidal functors $F,F':\G\to
\G'$ which are related by a monoidal natural isomorphism,
$\delta:F\overset{_\sim\ }\Rightarrow F'$, are considered the same.

The particular case of {\em categorical groups}  is well known since
it was dealt with by Sinh in 1975. Recall that a categorical group
\cite{Jo-St} (also called a {\em Gr-category} \cite{Br,Si} and a
{\em weak $2$-group} \cite{B-L}) is a monoidal groupoid $\G$ in
which every object $X$ is invertible, in the sense that there is
another object $X^*$ and an isomorphism $X\otimes X^*\cong \I$. In
\cite{Si}, Sinh proved that, for any group $G$, any $G$-module
${(A,\theta:G\to\mathrm{Aut}(A))}$, and any Eilenberg-Mac Lane
cohomology class $c\in\HH^3(G,(A,\theta))$, there exists a
categorical group $\G$, unique up to monoidal equivalence, such that
$G$ is the group of isomorphism classes of objects of $\G$,
$A=\mathrm{Aut}_\G(\I)$ is the (abelian) group of automorphisms in
$\G$ of the unit object, and the $G$-action $\theta$ and the
cohomology class $c$ are canonically deduced from the functoriality
of the tensor and  the naturality and coherence of the constraints
of $\G$. This fact was historically relevant since it pointed out
the utility of categorical groups in homotopy theory: as
$\HH^3(G,(A,\theta))=\HH^3(BG,(A,\theta))$ is the 3rd cohomology
group of the classifying space $BG$ of the group $G$ with local
coefficients in $(A,\theta)$, for any triplet of data
$(G,(A,\theta),c)$ as above, there exists a path-connected
CW-complex $X$, unique up to homotopy equivalence, such that
$\pi_iX=0$ if $i\neq 1,2$, $\pi_1X=G$, $\pi_2X=A$ (as $G$-modules),
and $c\in \HH^3(G,(A,\theta))$ is the unique non-trivial Postnikov
invariant of $X$. Therefore, categorical groups arise as algebraic
homotopy 2-types of path-connected spaces. Indeed, strict
categorical groups (that is, categorical groups in which all the
structure constraints are identities and the monoid of objects is a
group) are the same as {\em crossed modules}, whose use in homotopy
theory goes back to Whitehead (1949) (see \cite{b-s} for the
history).

However, many illustrative examples such as the category
$\mathcal{A}z_R$ of central separable algebras over a commutative
ring $R$,  or the fundamental groupoid $\pi X$ of a Stasheff
$A_4$-space $X$ (a topological monoid, for instance), show the
ubiquity of monoidal groupoids in several branches of mathematics,
and therefore the interest of studying these categorical structures
in their own right. But the situation with monoidal groupoids is
more difficult than with categorical groups. Let us stress the two
main differences between the two situations. On one hand,  {\em the
structure induced by the tensor product on the set of connected
components of a monoidal groupoid is that of a  monoid} rather than
a  group as happens in the categorical group case. On the other
hand, if $\G$ is a categorical group, then the isotropy groups
$\mathrm{Aut}_\G(X)$, $X\in \mathrm{Ob}\G$, are all abelian and all
isomorphic to $\mathrm{Aut}_\G(\I)$, while {\em a monoidal groupoid
may have some isotropy groups that are not isomorphic to
$\mathrm{Aut}_\G(\I)$, as well as some noncommutative isotropy
groups}. Think of the simple example  $\mathfrak{F}in$ of finite
sets and bijective functions between them, whose monoidal structure
is given by disjoint union construction: its monoid of isomorphism
classes of objects is $\mathbb{N}$, the additive monoid of natural
numbers, and its isotropy groups are the symmetric groups
$\mathfrak{G}_n$.

Strongly inspired by Schreier's analysis of group extensions
\cite{Sch} and its extension to fibrations of categories by
Grothendieck \cite{Gro} (but also by works of Sinh \cite{Si}, Breen
\cite{Br}, and others), the structure of the monoidal groupoids is
analyzed in this paper, where we develop a 3-dimensional
Schreier-Grothendieck factor set theory for monoidal groupoids,
which in fact involves a 2-dimensional one for the monoidal functors
between monoidal groupoids, and even a 1-dimensional one for the
monoidal transformations between them. More precisely, our general
conclusions on this issue concerning monoidal groupoids can be
summed up  by saying that we give explicit quasi-inverse
biequivalences
$$
\xymatrix{\mathbf{MonGpd}\ar@{}[r]|-{\thickapprox}
\ar@<4pt>[r]^-{\Delta}&\ar@<3pt>[l]^-{\Sigma} \mathbf{Z^3_{\mathrm{n\text{-}ab}}Mnd},
}
$$
between the 2-category of monoidal groupoids and the 2-category of
what we call {\em Schreier systems for monoidal groupoids}, or {\em
non-abelian $3$-cocycles on monoids}. That is, systems of data
$$(M,\AAw,\Theta,\lambda)$$  consisting of a monoid $M$, a
family of groups ${\AAw=(A_{a})_{a\in M}}$ parameterized by the
elements of the monoid, a family of group homomorphisms $$\Theta=(
A_b\overset{a_*}\longrightarrow A_{ab} \overset{b^*}\longleftarrow
A_a)_{a,b\,\in M},$$ and a normalized map
$$
\lambda:M\times M\times M\longrightarrow \cup_{_{a\in M}}A_a \ \mid\ \lambda_{a,b,c}\in A_{abc},
$$
satisfying various requirements. In the 2-category
$\mathbf{Z^3_{\mathrm{n\text{-}ab}}Mnd}$, every equivalence is
actually an isomorphism, so that our classification results are
effective.

The lower theory of group-valued non-abelian 2-cocycles on monoids by Leech \cite{leech} was extended to small categories by Wells in \cite{wells-2}. Interestingly, the work by Wells
suggests\footnote{ We thank to the referee for this observation.}  it is possible to develop a theory of non-abelian
3-cocycles on small categories, which is attractive to us,  to generalize our results
 about monoidal groupoids to (Benabou) bicategories \cite{street} whose 2-cells
are all invertible, that is, whose hom-categories are
groupoids.

When we focus on the special case of {\em monoidal abelian
groupoids}, that is, monoidal groupoids $\G$ whose isotropy groups
$\Aut_\G(X)$, $X\in\mathrm{Ob}\G$, are all abelian, then our
classification results are stated in a more enjoyable and precise
way by means of Leech's cohomology theory of monoids \cite{leech}.
The biequivalences above restrict to quasi-inverse biequivalences
$$
\xymatrix{\mathbf{MonAbGpd}\ar@{}[r]|-{\thickapprox}
\ar@<4pt>[r]^-{\Delta}&\ar@<3pt>[l]^-{\Sigma} \mathbf{Z^3}\mathbf{Mnd},
}
$$
between the full 2-subcategory $\mathbf{MonAbGpd}\subseteq
\mathbf{MonGpd}$ of monoidal abelian groupoids and the full
2-subcategory $\mathbf{Z^3}\mathbf{Mnd}\subseteq
\mathbf{Z^3_{\mathrm{n\text{-}ab}}Mnd}$ given by those Schreier
systems $(M,\AAw,\Theta,\lambda)$ in which every group $A_a$ of
$\AAw$ is abelian. But, the pair $(\AAw,\Theta)$ that occurs in any
such Schreier system is just a coefficient system for Leech
cohomology groups $\HH^n(M,(\AAw,\Theta))$ of the monoid $M$, and
$\lambda\in Z^3(M,(\AAw,\Theta))$ is a normalized 3-cocycle. From
this observation, we achieve the classification both of the monoidal
abelian groupoids and of the  monoidal functors between them, by
means of the cohomology groups $\HH^3(M,(\AAw,\Theta))$ and
$\HH^2(M,(\AAw,\Theta))$. Although these results  are mainly of
algebraic interest, we would like to stress their potential interest
in homotopy theory since, as we observe in the paper, there are
natural isomorphisms $\HH^n(M,(\AAw,\Theta))\cong
\HH^n(BM,(\AAw,\Theta))$, between Leech cohomology groups of a
monoid $M$ and Gabriel-Zisman's cohomology groups of the classifying
space $BM$ of the monoid with twisted coefficients in
$(\AAw,\Theta)$ \cite[Appendix II]{G-Z}, \cite[Chapter VI]{Ill}.

The plan of the paper, briefly, is as follows. After this
introductory Section 1, the paper is organized in four sections.
Section 2 comprises some notations and basic results concerning
monoidal groupoids and the 2-category that they form, as well as a
list of some striking examples of them. The main Section 3 includes our
`Schreier-Grothendieck theory' for monoidal groupoids. This is
quite a long and technical section, but crucial to our conclusions,
where we describe the 2-category
$\mathbf{Z^3_{\mathrm{n\text{-}ab}}Mnd}$ of non-abelian 3-cocycles
on monoids, and we show in detail how this 2-category is
biequivalent to the 2-category $\mathbf{MonGpd}$ of monoidal
groupoids. Section 4 focuses on the special case of  monoidal
abelian groupoids. In the first subsection, we briefly review some
aspects concerning Leech cohomology of monoids
$\HH^n(M,(\AAw,\Theta))$, and we observe how this cohomology theory is
actually a particular case of Gabriel-Zisman cohomology theory
 of the classifying space $BM$ of the monoid $M$. In the
second subsection, we include our main classification results
concerning  monoidal abelian groupoids in terms of Leech
cohomology groups. And, finally, a third subsection is devoted to
revisiting the 2-category of categorical groups in order to show how
the results obtained here imply the classification results already known for them.

\section{Preliminaries: The 2-category of monoidal groupoids} This section aims to make this paper as self-contained as possible; hence, at
the same time as fixing notations and terminology, we also review
some necessary aspects and results from the background of monoidal
categories that will be used throughout the paper. However, the
material here is quite standard, so the expert reader may skip most
of it.

A {\em monoidal  category}
$\M=(\M,\otimes,\mathrm{I},\boldsymbol{a},\boldsymbol{l},\boldsymbol{r})$
consists of a category  $\M$, a functor $$\otimes:\M\times \M\to
\M,\hspace{0.5cm}(X,Y)\mapsto XY$$ (the {\em tensor product}), a
distinguished object $\I\in \M$ (the {\em unit object}), and natural
isomorphisms
$$\xymatrix@C=14pt{\boldsymbol{a}_{X,Y,Z}:(XY)Z\ar[r]\ar@{}@<-2.5pt>[r]^(.6){\sim}|(.6){-}&X(YZ),}\hspace{0.5cm}
\xymatrix@C=14pt{\boldsymbol{l}_X:\I X\ar[r]\ar@{}@<-2.5pt>[r]^(.6){\sim}|(.6){-}& X,} \hspace{0.5cm}
\xymatrix@C=14pt{\boldsymbol{r}_X:X \I\ar[r]\ar@{}@<-2.5pt>[r]^(.6){\sim}|(.6){-}& X,}$$
(called the {\em associativity, left unit, and right unit constraints},
respectively), such that, for all objects $X,Y,Z,T$ of $\M$, the
diagrams below (called the {\em associativity pentagon} and the {\em
triangle for the unit}) commute.
\begin{equation}\label{mac1}\begin{array}{c}
\xymatrix@C=18pt@R=18pt{((XY)Z)T
\ar[r]^{\boldsymbol{a}}\ar[d]_{ \boldsymbol{a}1}&(XY)(ZT) \ar[r]^{\boldsymbol{a}}& X(Y(ZT))
\\ (X(YZ))T\ar[rr]^{\boldsymbol{a}}&&X((YZ)T)\ar[u]_{1\boldsymbol{a}}}\hspace{0.6cm}
\xymatrix@C=-2pt@R=18pt{(X \I)Y\ar[rr]^{\boldsymbol{a}}\ar[rd]_{\boldsymbol{r} 1}&&X(\I Y)\ar[ld]^{1 \boldsymbol{l}}\\&XY & }
\end{array}
\end{equation}

A monoidal category is called {\em strict} when all the constraints
$\boldsymbol{a}$, $\boldsymbol{l}$, $\boldsymbol{r}$ are identity
arrows.

 Observe that we usually write the structure constraints without
object labels, since their source and target  make it clear what
constraint isomorphism it is.

In any monoidal category, for any objects $X,Y$, the triangles below
commute \cite[Proposition 1.1] {Jo-St}.
\begin{equation}\label{tm2tr}
\begin{array}{c}
\xymatrix@C=-2pt@R=14pt{(XY)\I\ar[rr]^{\boldsymbol{a}}\ar[rd]_{\boldsymbol{r}}&&X(Y\I)\ar[ld]^{1\boldsymbol{r}}\\&XY&
}\hspace{0.4cm}\xymatrix@C=-2pt@R=14pt{(\I X)Y\ar[rr]^{\boldsymbol{a}}\ar[rd]_{\boldsymbol{l}1}&&\I
(XY)\ar[ld]^{\boldsymbol{l}}\\&XY& }
\end{array}
\end{equation}

If $\M,\M'$ are  monoidal categories, then a {\em monoidal functor}
${F=(F,\varphi):\M\to \M'}$ consists of a functor $F:\M\to \M'$, a
family of  natural isomorphisms
$$
\xymatrix@C=14pt{\varphi_{X,Y}:FX\,FY\ar[r]\ar@{}@<-2.5pt>[r]^(.58){\sim}|(.58){-}& F(XY),}
$$
and an isomorphism
$\xymatrix@C=14pt{\varphi_0:\I'\ar[r]\ar@{}@<-2.5pt>[r]^(.5){\sim}|(.5){-}&
F\I,}$ such that the following diagrams commute:
\begin{equation}\label{ecf}\begin{array}{c}
\xymatrix@R=18pt{ (FX\,FY)FZ\ar[r]^{ \varphi 1}\ar[d]_{\boldsymbol{a}'}&F(XY)FZ\ar[r]^{\varphi}&
F((XY)Z)\ar[d]^{F\boldsymbol{a}}\\
FX( FY\,FZ)\ar[r]^{1 \varphi}&FX\, F(YZ) \ar[r]^{\varphi}&F(X(YZ))}\end{array}
\end{equation}
\begin{equation}\label{ecf2}\begin{array}{c}
\xymatrix@R=18pt{FX\, \I'\ar[d]_{\boldsymbol{r}'}\ar[r]^{1 \varphi_0}&FX\,F\I\ar[d]^{\varphi}\\
FX&F(X\I),\ar[l]_{F\boldsymbol{r}}}\end{array}\hspace{0.4cm}
\begin{array}{c}
\xymatrix@R=18pt{\I'\,FX\ar[d]_{\boldsymbol{l}'}\ar[r]^{\varphi_01}&F\I\, FX\ar[d]^{\varphi}\\
FX&F(\I X),\ar[l]_{F\boldsymbol{l}}}\end{array}
\end{equation}

When $F\I=\I'$ and $\varphi_0=1_{\I'}$, the identity, then the
monoidal functor $F$ is qualified as {\em strictly unitary}. When
each of the isomorphisms $\varphi_{X,Y}$, $\varphi_0$ is an
identity, the monoidal functor is called {\em strict}.

The composition of monoidal functors $\M \overset{F}\to \M'
\overset{F'}\to \M''$ will be denoted by juxtaposition, that is,
$F'F:\M \to \M''$. Recall that its structure constraints are
obtained from those
 of $F$ and $F'$, by the
compositions
$$\begin{array}{l}
\xymatrix{F'FX\ F'FY\ar[r]^{\varphi'}&F'(FX\ FY)
\ar[r]^{F'\!\varphi}&F'F(XY)},\\
\xymatrix{\I''\ar[r]^-{\varphi'_0}&F'\I'\ar[r]^-{F'\!\varphi_0}&
F'F\,\I.}
\end{array}$$
The composition of monoidal functors is associative and unitary, so
that the category $\mathbf{MonCat}$ of
 monoidal categories is defined. Actually, this
is the underlying category of a 2-category, also denoted by
$\mathbf{MonCat}$, whose 2-arrows are the {\em morphisms} of
monoidal functors or {\em monoidal natural transformations}. If
$F,F':\M\to\M'$ are monoidal functors, then a morphism between them
is a natural transformation on the underlying functors
$\delta:F\Rightarrow F'$, such that, for all objects  $X,Y$ of $\M$,
the following coherence diagrams commute:
\begin{equation}\label{cfnt}\begin{array}{c}
\xymatrix@R=18pt{FX\, FY\ar[r]^{\varphi}\ar[d]_{\delta_X \delta_Y}&F(XY)\ar[d]^{\delta_{XY}}\\
F'\!X\,F'Y\ar[r]^{\varphi'}&F'(XY)}\hspace{0.6cm} \xymatrix@C=10pt@R=20pt{&\I'\ar[rd]^{\varphi'_0}
\ar[dl]_{\varphi_0}&\\ F\I\ar[rr]^{\delta_{\I}}&&F'\I}\end{array}
\end{equation}

In this 2-category, the ``vertical composition'' of 2-cells, denoted
by
$$
\xymatrix@R=20pt@C=14pt{&\ar@{}@<-4pt>[d]^(.7){\Downarrow\delta} &\\
\M\ar@/^1.5pc/[rr]^{F}\ar@/_1.5pc/[rr]_{F''}\ar[rr]|-{F'}&\ar@{}@<-4pt>[d]^(.3){\Downarrow\delta'} &
\M'\\
&& }\xymatrix@R=20pt@C=10pt{&\\ \ar@{|->}@<-2pt>[r]^{ \circ}&\\&}
\xymatrix@R=20pt@C=14pt{&\ar@{}[dd]|{\Downarrow\,\delta'\circ\,\delta} &\\
\M\ar@/^1pc/[rr]^{F}\ar@/_1pc/[rr]_{F''}& &\M',\\
&& }
$$
is given by the ordinary vertical composition of natural
transformations, that is, the component of $\delta'\circ \delta$ at
any object $X$ of $\M$ is given by the composition in $\M'$
\begin{equation}\label{pspps}(\delta'\circ\delta)_X=\delta'_X\circ\delta_X:FX\overset{\delta_X}
\longrightarrow
F'\!X\overset{\delta'_X}\longrightarrow F''\!X.\end{equation} Similarly, the
``horizontal composition''
$$\xymatrix @C=8pt{\M\ar@/^0.8pc/[rr]^{F}\ar@/_0.8pc/[rr]_{G}\ar@{}[rr]|{\Downarrow\delta}& &
 \M'\ar@/^0.8pc/[rr]^{F'}\ar@/_0.8pc/[rr]_{G'}\ar@{}[rr]|{ \Downarrow\delta'}&&
   \M''} \mapsto \xymatrix@C=8pt{\M\ar@/^0.9pc/[rr]^{ F'F}
 \ar@/_0.9pc/[rr]_{G'G}\ar@{}[rr]|{\Downarrow{\delta'\delta}} && \M''},$$
is given by the usual horizontal composition of natural
transformations:
\begin{equation}\label{psps}
\delta'\delta =G'\delta\circ \delta'F=\delta'G\circ F'\delta:F'F\Rightarrow G'G.
\end{equation}

The following known lemma will be useful in the sequel (cf.
\cite[Lemma 1.1]{ceg02II}, for example). Let
$$\mathbf{MonCat}_{\mathrm u}\subseteq \mathbf{MonCat}$$ denote the
2-subcategory of the 2-category of monoidal categories which is full
on 0-cells and 2-cells,  but whose 1-cells are the strictly unitary
monoidal functors.

\begin{lemma} \label{uslem}
The inclusion  $\mathbf{MonCat}_{\mathrm
u}\hookrightarrow
\mathbf{MonCat}$ is a biequivalence.
\end{lemma}
\begin{proof}
For any two monoidal categories $\M$ and $\M'$, a quasi-inverse to
the inclusion functor $i:\mathbf{MonCat}_{\mathrm
u}(\M,\M')\hookrightarrow \mathbf{MonCat}(\M,\M')$,
\begin{equation}\label{norm}
(\ )^{\mathrm u}: \mathbf{MonCat}(\M,\M')\to \mathbf{MonCat}_{\mathrm u}(\M,\M'),
\end{equation}
which should be called the {\em normalization functor}, works as
follows:  For any given monoidal functor $F=(F,\varphi):\M\to \M'$,
let $\Psi_F=(\psi_X)_{X\in \mathrm{Ob}\M}$ be the family of
isomorphisms in $\M'$
$$
\psi_X=\left\{\begin{array}{lll} 1_{FX}:FX\to FX&\text{if}& X\neq I\\[4pt]
\varphi_0^{-1}:F\I\to \I'&\text{if}& X=I.\end{array}\right.
$$
Then, $F$ can be deformed to a new monoidal functor, $F^{\mathrm
u}=(F^{\mathrm u},\varphi^{\mathrm u}):\M\to \M'$, in a unique way
such that $\Psi_F:F\overset{_\sim\ }\Rightarrow F^{\mathrm u}$
becomes an isomorphism. Namely,
$$
\begin{array}{l}F^{\mathrm u}\!X=\left\{\begin{array}{lll} FX& \text{if}& X\neq I\\[4pt]
\I'&\text{if}& X=I,\end{array}\right.\hspace{0.4cm}
F^{\mathrm u}(X\overset{f}\to Y)=\xymatrix@C=50pt{(F^{\mathrm u}\!X\ar[r]^{\psi_Y\circ
 Ff\circ \psi_X^{-1}}&F^{\mathrm u}Y),}
\\[14pt]
\varphi^{\mathrm u}_{X,Y}=\psi_{XY}\circ \varphi_{X,Y}\circ (\psi_X\psi_Y)^{-1},\hspace{0.4cm}
\varphi^{\mathrm u}_0=\psi_I\circ \varphi_0=1_{\I'}.
\end{array}
$$

 Furthermore, any morphism $\delta:F\Rightarrow G$  gives rise to
 the morphism
$$\delta^{\mathrm u}=\Psi_G^{-1}\circ\delta\circ\Psi_F:F^{\mathrm u}\Rightarrow G^{\mathrm u},$$
which is explicitly given by
$$
\delta^{\mathrm u}_X=\left\{\begin{array}{lll} \delta_{X}:FX\to GX&\text{if}& X\neq I\\[4pt]
\varphi_0\circ \delta_\I\circ \varphi_0^{-1}=1_{\I'}:\I'\to \I'&\text{if}& X=I.\end{array}\right.
$$

These mappings $F\mapsto F^{\mathrm u}$, $\delta\mapsto
\delta^{\mathrm u}$, describe the normalization functor
$(\ref{norm})$.

Since, by construction,  $(\ )^{\mathrm u}\, i=1$, the identity
functor, and we have the natural isomorphism $\Psi:1\overset{_\sim\
}\Rightarrow i\, (\ )^{\mathrm u}$, $F\mapsto \Psi_F$, both functors
$i$ and $(\ )^{\mathrm u}$ are mutually quasi-inverse.\qed
\end{proof}

A monoidal functor $F:\M\to \M'$ is called a {\em monoidal
equivalence} when there exists a monoidal  functor $F':\M'\to \M$
and isomorphisms of monoidal functors $1_\M\overset{_\sim\
}\Rightarrow F'F$ and $FF'\overset{_\sim\ }\Rightarrow1_{\M'}$. Two
monoidal categories are {\em equivalent} if they are connected by a
monoidal equivalence. From Saavedra \cite[I, Proposition 4.4.2]{Sa},
we have the following useful result:
\begin{proposition}\label{saa} A monoidal functor $(F,\varphi):\M\to \M'$ is
a monoidal equivalence if and only if the underlying functor
$F:\M\to \M'$ is an equivalence of categories; that is, if and only
if the functor $F$ is full, faithful, and each object of $\M'$ is
isomorphic to an object of the form $FX$ for some $X\in \M$.
\end{proposition}

In this paper,  we are mainly going to work with the full
2-subcategory of $\mathbf{MonCat}$ given by the monoidal groupoids, that is, of monoidal categories whose morphisms are all
invertible, hereafter
denoted by
$$
\mathbf{MonGpd}.
$$

This 2-category of monoidal groupoids contains as a full 2-subcategory
the better-known 2-category of categorical groups, denoted by $$\mathbf{CatGp},$$
whose objects, recall, are those monoidal groupoids in which every object is invertible. The inclusion $\mathbf{CatGp}\hookrightarrow\mathbf{MonGpd}$ has a right biadjoint 2-functor $$\mathcal{P}ic:\mathbf{MonGpd}\to \mathbf{CatGp}$$
that assigns to each monoidal groupoid $\G$ its {\em Picard categorical group} \cite[2.5.1]{Sa},
$$\mathcal{P}ic(\G)\subseteq \G,$$
which is defined as the monoidal full subgroupoid of $\G$ given by the invertible objects.

\subsection{Examples} To help motivate the reader, we shall show some classic and striking
instances of monoidal groupoids. The most basic example of a
monoidal groupoid is perhaps that defined by the category
$\mathfrak{F}in$ of finite sets and bijective functions between
them, whose monoidal structure is given by means of the disjoint
union construction, which arises in the study of categories of
representations of the symmetric groups $\mathfrak{S}_n$ (see Joyal
\cite{J}). Indeed, $\mathfrak{F}in$ is equivalent to the strict
monoidal groupoid $\mathfrak{G}$ defined as the disjoint union of
the symmetric groups $\mathfrak{S}_n$, $n\in \mathbb{N}$.  More
precisely, $\mathfrak{G}$ has objects the natural numbers $n\in
\mathbb{N}$, and the hom-sets are given by
$$
\mathfrak{G}(m,n)=\left\{\begin{array}{ll}\mathfrak{G}_n& \text{ if } m=n\\
\emptyset &\text{ if } m\neq n\,.\end{array}\right.
$$
Composition is multiplication in the symmetric groups,  and the
tensor product is given by the obvious map $\mathfrak{G}_m\times
\mathfrak{G}_n\to \mathfrak{G}_{m+n}$.

Ring theory is a good source of many interesting monoidal groupoids.
 For example, following Fr\"{o}hlich and Wall \cite{fw3}, let $R$ be any given commutative ring. Then, the
monoidal category of $R$-modules, $\M od_R$, whose monoidal
structure is given by the usual tensor product of $R$-modules,
$(M,N)\mapsto M\otimes_RN$, contains as an interesting  monoidal
subcategory the so-called {\em  monoidal groupoid of
$R$-progenerators},   usually denoted  by $$\G en_R,$$ whose objects
are the faithful, finitely generated projective $R$-modules, and
whose morphisms are the module isomorphisms between them. The invertible objects in
$\G en_R$ are the invertible $R$-modules, that is, rank 1 projectives.
Therefore,
$$\mathcal{P}ic(\G en_R) =\mathcal{P}ic_R,$$ is the monoidal groupoid  known as the
 {\em Picard categorical group of $R$.}
Similarly, the monoidal category of associative $R$-algebras with
identity,   $\mathcal{A}lg_R$, whose monoidal structure is given by
the ordinary tensor product of $R$-algebras, $(A,B)\mapsto
A\otimes_RB$, contains a striking instance of a monoidal groupoid:
the so-called {\em monoidal groupoid of Azumaya $R$-algebras},
denoted by $$\mathcal{A}z_R,$$ whose objects are the central
separable $R$-algebras and whose morphisms are the $R$-algebra isomorphisms. Forgetting algebra structure and taking the endomorphism ring
define, respectively, two remarkable monoidal functors:
$\mathrm{Lin}_R:\mathcal{A}z_R\to \G en_R$ and $\mathrm{End}_R:\G
en_R\to \mathcal{A}z_R$. The {\em Morita monoidal groupoid of
$R$-algebras},
$$\M\mathcal{A}lg_R,$$ has objects $R$-algebras, and a morphism
$A\to B$ is an isomorphism class of a Morita equivalence $\M
od_A\overset{_\thicksim\ }\to\M od_B$ (or, equivalently, an
isomorphism class of an invertible (left) $A\otimes_RB^{\mathrm
{op}}$-module). An object $A$ of this monoidal groupoid is
invertible if and only if there is another object $B$ such that
$A\otimes_RB$ is Morita equivalent to $R$. It follows that $A$ must
be an Azumaya $R$-algebra. Conversely, if $A$ is Azumaya, the
isomorphism $A\otimes_RA^{\mathrm{op}}\cong
\mathrm{End}_R(\mathrm{Lin}_R(A))$ shows that, since
$\mathrm{End}_R(\mathrm{Lin}_R(A))$ is Morita equivalent to $R$,
$A^\mathrm{op}$ provides a quasi-inverse of $A$. Hence,
$$\mathcal{P}ic(\M \mathcal{A}lg_R)=\mathcal{B}r_R$$ is the {\em Brauer categorical
group of $R$}, whose objects are the same as those of
$\mathcal{A}z_R$, that is,  the Azumaya $R$-algebras, but whose morphisms
are here iso-classes of Morita equivalences between them.

Every monoidal groupoid arises from an elemental categorical
construction: If $\mathcal{B}$ is any bicategory \cite{street}, then
the {\em monoidal groupoid of endomorphisms of an object} $b\in
\mathcal{B}$, denoted by
$$\mathcal{E}nd(b),$$ has objects the 1-cells $f:b\to b$ in
$\mathcal{B}$ and morphisms the invertible 2-cells
${f\overset{_\thicksim\ }\Rightarrow f'}$ between them. The monoidal
structure on $\mathcal{E}nd(b) $ is given by the horizontal
composition of cells in the bicategory. The {\em categorical group
of autoequivalences} of $b$ is
$$\mathcal{A}ut(b)=\mathcal{P}ic(\mathcal{E}nd(b)),$$ that is, the
monoidal full subgroupoid of equivalences $b\overset{_\thicksim\,
}\to b$ in the bicategory. If, for example, we take
$\B=\mathbf{Cat}$, the 2-category of categories, and  ${\mathcal C}$
is any category, then the  monoidal groupoid
$$\mathcal{E}nd(\mathcal{C})$$ has objects the functors
$F:\mathcal{C}\to \mathcal{C}$ and the morphisms are the natural
isomorphisms $F\overset{_\sim\ }\Rightarrow G$. The composition in
$\mathcal{E}nd(\mathcal{C})$ is given by the usual vertical
composition of natural transformations, while the composition of the
functors and the horizontal composition of the natural
transformations define its (strict) monoidal structure. These
monoidal groupoids of endofunctors are relevant in several
frameworks since a pseudo-action of a monoidal category $\M$ on a
category $\mathcal{C}$ is the same thing as a monoidal functor
$\M\to \mathcal{E}nd(\mathcal{C})$. For instance, a Deligne action
\cite{Del} of a monoid $M$ (say, for example,
$M=\mathbf{B}^+(\mathbf{W},S)$ the Artin-Tits monoid of positive
braids defined by a finite Coxeter group $\mathbf{W}$ with a set of
generators $S$) on a category $\mathcal{C}$, is just a monoidal
functor $M\to \mathcal{E}nd(\mathcal{C})$, from the discrete
monoidal category that $M$ defines to the monoidal groupoid of
endofunctors of $\mathcal{C}$.

The {\em Picard categorical group of a category} $\mathcal{C}$ is
$$\mathcal{P}ic(\mathcal{C})=\mathcal{A}ut(\mathcal{C}),$$
that is, the monoidal full subgroupoid of
$\mathcal{E}nd(\mathcal{C})$ given by the autoequivalences
$\mathcal{C}\overset{_\thicksim\, }\to \mathcal{C}$. If, for
example, $A$ is any ring and we take $\mathcal{C}={}_A\M od_A$, the
category of $A$-bimodules, then, by Morita's theory, there is a
monoidal equivalence
$$
\mathcal{A}ut({}_A\M od_A)\simeq \mathcal{P}ic_A,
$$
where $\mathcal{P}ic_A$ is the Picard categorical group of the ring,
that is, the categorical group of invertible $A$-bimodules with
isomorphisms, whose monoidal structure is given by the usual
monoidal product of $A$-bimodules $(M,N)\mapsto M\otimes_AN$. The
case where $\mathcal{C}=G$ a group regarded as a category with only
one object, is also well-known. The monoidal groupoid
$$\mathcal{E}nd(G)$$
can be described  as having objects the group of endomorphisms
$f:G\to G$ and morphisms $u:f\to f'$ those elements $u\in G$
such that $f=C_uf'$, where ${C_u:G\to G}$ is the inner automorphism
$C_u(v)=uvu^{-1}$ given by conjugation with $u$. The composition of
morphisms is multiplication in $G$, and the (strict) monoidal
structure is defined by
$$
(f\overset{u}\to f')\otimes (g\overset{v}\to g')=\xymatrix@C=30pt{(fg\ar[r]^{u\,f'(v)}&f'g').}
$$
The corresponding Picard categorical group of invertible
objects
$$
\mathcal{A}ut(G),
$$
is the {\em categorical group of automorphisms of $G$}. It is the
internal groupoid in the category of groups whose group of objects
is $\mathrm{Aut}(G)$, the group of automorphisms of $G$,  and whose
group of arrows is the holomorph group $\mathrm{Hol}(G)=G\rtimes
\mathrm{Aut}(G)$. Thus, $\mathcal{A}ut(G)$ is precisely the
categorical group corresponding to the universal crossed module
$G\overset{C}\to \mathrm{Aut}(G)$ by the well-known Verdier
equivalence between the category of Whitehead crossed modules and
the category of strict categorical groups, see \cite{b-s} for the
history.

Algebraic Topology is also a natural setting where monoidal
groupoids appear with recognized interest. Recall that the
fundamental groupoid $\pi X$, of a space $X$, is the category having
$X$ as set of objects, and whose morphisms $[\omega]:x\to y$
($x,y\in X$) are relative end points homotopy classes of paths
$\omega:[0,1]\to X$ with $\omega(0)=x$ and $\omega(1)=y$.
The composition in $\pi X$ is induced by the usual concatenation of
paths and constant paths provide the identities.  Any continuous map
$f:X\to Y$ induces a functor $f_*:\pi X\to \pi Y$ given by
$$f_*\big(x\overset{[\omega]}\longrightarrow y\big)\,=
\,\big(f(x)\overset{[f\omega]}\longrightarrow f(y)\big),$$  so that
the fundamental groupoid construction, $X\mapsto \pi X$,  is a
functor from the category of topological spaces to the category of
groupoids. If $f,g:X\to Y$ are two maps, then a homotopy
$\alpha:f\Rightarrow g$,  $\alpha:[0,1]\to Y^X$, induces a natural
isomorphism $\alpha_*:f_*\Rightarrow g_*$ defined, for any point
$x\in X$, by
$$\alpha_*(x)=[\alpha(-)(x)]:f(x)\to g(x).$$ Moreover,  it is easy to see
that if two homotopies $\alpha,\beta: f\Rightarrow g$  are related
by a relative end maps homotopy, $\alpha\Rrightarrow \beta$, then
both induce the same natural isomorphism, that is, if
$[\alpha]=[\beta]$ in the track groupoid $\pi Y^X$, then
$\alpha_*=\beta_*:f_*\Rightarrow g_*$.

Suppose now that $X=(X,m,e,\alpha,\lambda,\rho)$ is any given
homotopy-coherent associative $H$-space, that is,  a Stasheff $A_4$-space
\cite{Stas} (any topological monoid, for instance). This means that
we have a topological space $X$, which is endowed with a continuous
multiplication map  $m:X\times X\to X$, a point $e\in X$, and
homotopies
$$\begin{array}{ll}
\xymatrix{X\times X\times X\ar@<-13pt>@{}[r]|(.47){\alpha}\ar@{}@<-18pt>[r]|(.5){\Rightarrow}\ar[r]^-{1\times m}\ar[d]_{m\times 1}&X\times X\ar[d]^{m}\\X\times X\ar[r]_{m}& X,}&
\xymatrix{X\ar@<-10pt>@{}[r]|(.73){\lambda}\ar@{}@<-16pt>[r]|(.7){\Leftarrow}\ar[rd]|-{\,1}\ar[r]^{e\times 1}\ar[d]_{1\times e}&X\times X\ar[d]^{m} \\ X\times X\ar@<14pt>@{}[r]|(.3){\rho}\ar@{}@<10pt>[r]|(.35){\Rightarrow}\ar[r]_{m}&X,}
\end{array}
$$
which are homotopy coherent, in the sense that there are homotopies
as below.
$$
\xymatrix@C=10pt{&X^4\ar[ld]_{1\times 1\times m}\ar[rd]|{1\times m\times 1}\ar[rr]^{m\times 1\times 1}&&X^3\ar@{}@<-8pt>[dd]|(.25){\Downarrow \alpha\times 1}|(.75){\Downarrow\alpha}\ar[rd]^{m\times 1}&&&&X^4\ar@{}[rd]|{=}\ar[ld]_{1\times 1\times m}\ar[rr]^{m\times 1\times 1}&&X^3\ar[ld]|{1\times m}\ar[rd]^{m\times 1}&
\\
X^3\ar[rd]_{1\times m}&&X^3\ar@<4pt>@{}[ll]|{\Leftarrow}\ar@<-2pt>@{}[ll]|{1\times \alpha}\ar[ld]^{1\times m}\ar[rr]^{m\times 1}&&X^2\ar[ld]^{m}&\Rrightarrow&X^3\ar[rd]_{1\times m}\ar[rr]^{m\times 1}&&X^2\ar[rd]_{m}&&X^2\ar[ld]^{m}\ar@<2pt>@{}[ll]|{\Leftarrow}\ar@<-4pt>@{}[ll]|{ \alpha}
\\
&X^2\ar[rr]_{m}&&X&&&&X^2\ar@{}[ru]|{\Downarrow\alpha}\ar[rr]_{m}&&X&
}
$$
$$
\xymatrix@R=8pt@C=25pt{
X^2\ar@<20pt>@{}[d]|(.6){\Uparrow}\ar@<30pt>@{}[d]|(.6){1\times\lambda}\ar[dd]_{1\times e\times 1}\ar[rr]^1&
\ar@<20pt>@{}[dd]|(.6){\Uparrow}\ar@<25pt>@{}[dd]|(.6){\alpha}
 &X^2
\ar[dd]^{m}&&X^2\ar@<24pt>@{}[d]|(.6){\Uparrow}\ar@<33pt>@{}[d]|(.6){\rho\times1}\ar[dd]|{1\times e\times 1}\ar[rr]^1&
\ar@<20pt>@{}[dd]|(.6){=} & X^2\ar[dd]^{m}\\
&&&\Rrightarrow&&&\\
X^3\ar[rruu]|(.45){1\times m}\ar[r]_{m\times 1}&X^2\ar[r]_{m}&X&&X^3\ar[rruu]|(.45){m\times 1}\ar[r]_{m\times 1}&X^2\ar[r]_{m}&X
}
$$
Since the functor $X\mapsto \pi X$ preserves products, the
multiplication map ${X\times X\to X}$ induces a tensor product
$$m_*:\pi X\times \pi X\cong\pi(X\times X)\longrightarrow \pi X,$$
and the homotopies $\alpha$, $\lambda$, and $\rho$, induce
corresponding associativity, left unit, and right unit constraints
(which satisfy the pentagon and triangle axioms $(\ref{mac1})$
thanks to the existence of the homotopies $\Rrightarrow$ above), we
have thus defined {\em the fundamental monoidal groupoid of the
$H$-space}
$$
\pi X=(\pi X, m_*,e,\alpha_*, \lambda_*, \rho_*).
$$

Let us stress that $\pi X$ is a categorical group whenever $X$ is
group-like  (for instance if $X\simeq \Omega(Y,y_0)$ is any loop
space).

\section{Schreier-Grothendieck theory for monoidal groupoids}
The Schreier Extension Theorem \cite{Sch} gives a cohomological
classification of extensions of (non-abelian) groups, $1\to A\to
E\to G\to 1$,
 in terms of
equivalence classes of the so-called {\em Schreier systems for group
extensions} or {\em non-abelian $2$-cocycles on groups}. That is, by
means of systems of data
\begin{equation}\label{schg}
(G,A,\Theta,\lambda),
\end{equation}
consisting of groups $G$ and $A$, a family of automorphisms $\Theta
=(A\overset{a_*}\longrightarrow A)_{a\in G}$, and a family of
elements $\lambda=(\lambda_{a,b}\in A)_{a,b\in G}$, satisfying
$$\begin{array}{cc}\lambda_{a,b}\circ (ab)_*(f) \circ
\lambda_{a,b}^{-1}=a_*(b_*(f)), & 1_*(f)=f,\\[5pt]
a_*(\lambda_{b,c})\circ \lambda_{a,bc} =\lambda_{a,b}\circ
\lambda_{ab,c},&\lambda_{a,1}=1=\lambda_{1,a},\end{array}$$ where
$f$ is any element of the group $A$. Any such Schreier system gives
rise to a group extension   \begin{equation}\label{gpext} 1\to A\to
\Sigma(G,A,\Theta, \lambda)\to G\to 1,\end{equation} where
$\Sigma(G,A,\Theta, \lambda)$ is the group defined by considering on
the cartesian product set $A\times G$ the multiplication $(f,a)\circ
(g,b)=(f\circ a_*(g)\circ \lambda_{a,b},ab)$, and any group
extension can be obtained in this way up to isomorphism. Actually,
the construction of the group extension $(\ref{gpext})$, from each
Schreier system $(\ref{schg})$, defines the function on objects of
an equivalence of categories between the category of Schreier
systems for group extensions, whose morphisms
$$(p,q,\varphi):(G,A,\Theta, \lambda)\to (G',A',\Theta', \lambda')$$
are triplets consisting of homomorphisms $p:G\to G'$, $q:A\to A'$,
and a family of elements $\varphi=(\varphi_a\in A')_{a\in G}$,
satisfying:
$$
\begin{array}{cc} \varphi_a\circ p(a)_*(q(f))\circ \varphi_a^{-1}= q(a_*(f)),\\[5pt]
q(\lambda_{a,b}) \circ \varphi_{ab}=\varphi_a\circ p(a)_*(\varphi_b)\circ \lambda'_{p(a),p(b)},\end{array}
$$
and the category of extensions of groups,  whose morphisms are
commutative diagrams
$$
\xymatrix@C=15pt@R=15pt{1\ar[r]&G\ar[d]_q\ar[r]&E\ar[d]_\phi\ar[r]&H\ar[d]^p\ar[r]&1\\
1\ar[r]&G'\ar[r]&E'\ar[r]&H'\ar[r]&1.}
$$

Several generalizations to monoid extensions of Schreier theory are
known in the literature: R\'edey \cite{Re}, Leech
\cite{leech,leech2}, Inassaridze \cite{Inas}, and so on. To classify
fibrations between categories, Grothendieck \cite{Gro} raised Schreier's theorem to a
categorical level by means of the theory of
pseudo-functors, and higher analogue problems were studied, among
others, by Sinh in \cite{Si}, where she  performed the categorical
group classification; Breen \cite{Br}, who treated with non-abelian
3-cocycles of groups for the classification of extensions of groups
by categorical groups; Carrasco and Cegarra in \cite{car94}, where
they carried out the classification of central extensions of
categorical groups; Ulbrich \cite{Ul}, who classified extensions of
Picard categories; Cegarra and Garz\'on in \cite{c-g}, where
a classification of torsors over a category under a categorical
group is done; or Cegarra and Khmaladze \cite{c2,c3}, where
the classification of both braided and symmetric graded categorical
groups is performed, later extended to the fibred cases by Calvo, Cegarra and
Quang in \cite{c-c-q}. We are inspired in all these works to
 undertake a corresponding analysis of
monoidal groupoids below,  where we achieve a 3-dimensional
Schreier-Grothendieck factor set theory for the classification of
monoidal groupoids, which in fact involves a 2-dimensional one for
monoidal functors between monoidal groupoids, and even a
1-dimensional one for the monoidal transformations between them.

\subsection{Schreier systems for monoidal groupoids}
 Keeping the Schreier-Grothendieck theory in mind,
we introduce $3$-dimensional {\em Schreier systems for monoidal
groupoids}, or {\em non-abelian $3$-cocycles on monoids}, which will
be shown as appropriate minimal systems of ``descent datum" to build
a survey of all monoidal groupoids up to monoidal equivalences.

\begin{definition}\label{defSch} A {\em Schreier system (for a monoidal
groupoid)} $\mathcal{S}= (M,\AAw,\Theta,\lambda)$ consists of the
following data:
\begin{itemize}
\item[$\bullet$] a monoid $M$,
\item[$\bullet$] a family of groups $\AAw=(A_{a})_{a\in M}$,
\item[$\bullet$] a family of homomorphisms $\Theta=(
A_b\overset{a_*}\longrightarrow A_{ab}
\overset{b^*}\longleftarrow A_a)_{a,b\,\in M}$,
\item[$\bullet$] a family of elements $\lambda=(\lambda_{a,b,c}\in A_{abc})_{a,b,c \in M}$.
\end{itemize}

These data must satisfy the following seven conditions:

$\bullet$ For any $a,b,c\in M$, $h\in A_a$, $g\in A_b$, and $f\in
A_c$,
\begin{eqnarray}\label{c11}
\lambda_{a,b,c}\circ (ab)_*(f) \circ
\lambda_{a,b,c}^{-1}=a_*(b_*(f)),\,\\[5pt]\label{c12}
\lambda_{a,b,c}\circ c^*(a_*(g)) \circ
\lambda_{a,b,c}^{-1}=a_*(c^*(g)),\\[5pt]\label{c13}
\lambda_{a,b,c}\circ c^*(b^*(h)) \circ
\lambda_{a,b,c}^{-1}=(bc)^*(h). \ \
\end{eqnarray}

$\bullet$ For any $a,b,c,d\in M$,
\begin{equation}\label{c2} a_*(\lambda_{b,c,d})\circ \lambda_{a,bc,d}\circ d^*(\lambda_{a,b,c})=
\lambda_{a,b,cd}\circ\lambda_{ab, c,d}.
\end{equation}

$\bullet$ For any $a,b\in M$, $g\in A_a$, and $f\in A_b$,
\begin{equation}\label{c3}
a_*(f)\circ b^*(g)=b^*(g)\circ a_*(f).
\end{equation}

$\bullet$  For any $a\in M$ and $f\in A_a$,
\begin{equation}\label{c4} 1_*(f)=f=1^*(f).
\end{equation}

$\bullet$  For any $a,b\in M$,
\begin{equation}\label{c5}
\lambda_{1,a,b}=\lambda_{a,1,b}=\lambda_{a,b,1}=1.
\end{equation}
\end{definition}

\begin{example}\label{exagp}  A Schreier system as above with $\lambda=1$
(i.e., such that $\lambda_{a,b,c}=1$  for all $a,b,c\in M$) is the
same thing as a pair of data $(M,(\AAw,\Theta))$ consisting of a
monoid $M$ together with an internal group object $(\AAw,\Theta)\in
\mathbf{Gp}(\mathbf{Mnd}\!\downarrow_{\! M})$, in the comma category
 of monoids over $M$. We refer  to Wells \cite[Theorem 6]{wells} for details, but briefly
 let us say that, for that identification, one regards
 $(\AAw,\Theta)$ as the monoid obtained as the disjoint union of the groups
 $A_a$, $a\in M$,
  with multiplication given by  $(f,a)(g,b)= (a_*(f)\circ
b^*(g),ab)$.
   This multiplication
   is associative thanks to $(\ref{c11})$, $(\ref{c12})$, and $(\ref{c13})$, and
   it is unitary, with $(1,1)$ its unity, owing
   to $(\ref{c4})$. The monoid homomorphism $\cup_{a\in M} A_a\to\!M$ is the obvious projection
   $(f,a)\mapsto a$, and the internal group operation is defined by the
   map $$\cup_{a\in M} A_a\times_M\cup_{a\in M} A_a\to \cup_{a\in M} A_a, \hspace{0.4cm}
   ((f,a),(g,a))\mapsto (f\circ g,a),$$
   which is plainly recognized  to be a monoid homomorphism thanks to the centralizing
   condition $(\ref{c3})$.

Surjective monoid homomorphisms $E\to M$ endowed with a principal
homogenous internal $(\AAw,\Theta)$-action in
$\mathbf{Mon}\!\downarrow{\!_M}$ (i.e., internal
$(\AAw,\Theta)$-torsors) are classified by means of Leech {\em
non-abelian $2$-cocycles of $M$ with coefficients in
$(\AAw,\Theta)$}. That is, by families $\lambda=(\lambda_{a,b})$ of
elements $\lambda_{a,b}\in A_{ab}$, one for each $a,b\in M$, such
that
$$a_*(\lambda_{b,c})\circ \lambda_{a,bc}=c^*(\lambda_{a,b})\circ
\lambda_{ab,c}, \hspace{0.4cm}\lambda_{1,a}=1=\lambda_{a,1},$$ for
all $a,b,c\in M$; see Leech \cite[Section 3]{leech} and Wells
\cite[Theorems 1 and 7]{wells}.
\end{example}

\begin{remark}
Any monoid can be regarded as a small category with only one object,
and  natural generalizations to small categories of the above facts in Example \ref{exagp} can be found in the unpublished paper by Wells
\cite{wells-2}, where he
handles non-abelian coefficients as well as the abelian ones that the (commonly attributed to  Baues and Wirshing \cite{B-W}) cohomology of small categories theory uses. The results by Wells suggest that might be possible to develop a
theory of group-valued non-abelian 3-cocycles on small categories in order to
generalize our results in this paper about monoidal groupoids to
bicategories \cite{street} whose hom-categories are
groupoids, that is, to `many
objects' monoidal groupoids. Although this generalization (quite far from obvious) to
bicategories is beyond the scope and possibilities of this already long
paper, we shall point out that these {\em non-abelian $3$-cocycles
on  a category $\mathcal{C}$} should be systems of data consisting of
a group $A_a$ for each arrow $a:y\to x$ of $\mathcal{C}$, group
homomorphisms $A_b\overset{a_*}\longrightarrow A_{ab}
\overset{b^*}\longleftarrow A_a$ for each pair of composible arrows
$z\overset{b}\to y\overset{a}\to x$ in $\mathcal{C}$, and elements
$\lambda_{a,b,c}\in A_{abc}$, one for each three composible arrows
$t\overset{c}\to z\overset{b}\to y\overset{a}\to x$ in
$\mathcal{C}$, all subject to the corresponding seven conditions as in
$(\ref{c11})- (\ref{c5})$.
\end{remark}

\begin{remark} Regarding any group as a groupoid with exactly one object, it was observed by
Grothendieck \cite{Gro} that a non-abelian 2-cocycle
$(G,A,\Theta,\lambda)$ for a group extension of a group $G$ by a
group $A$, as in $(\ref{schg})$, can be identified as a normal
pseudo-functor on $G$ that associates the group $A$ to the unique
object of $G$. Similarly, as one identifies any monoid with the
monoidal discrete category it defines, then a Schreier system $(M,
\AAw,\Theta,\lambda)$ for a monoidal groupoid, as in Definition
\ref{defSch}, can be viewed as a group valued normal monoidal
pseudo-functor on  $M$, in the sense of Carrasco-Cegarra
\cite[Definition 1.6]{car94}, that associates the group $A_a$ to
each object $a\in M$.
\end{remark}

Next we explain how Schreier systems, as in Definition \ref{defSch},
are characteristically associated to monoidal groupoids.

\subsection{Schreier systems associated to  monoidal groupoids}\label{ssamg}
For any given monoidal groupoid
$\G=(\G,\otimes,\mathrm{I},\boldsymbol{a},\boldsymbol{l},\boldsymbol{r})$,
let
\begin{equation}\label{m}
M(\G)\!=\,^{_{\textstyle\mathrm{Ob}\G}}\!\!/_{{\!\cong}}
\end{equation}
be the monoid of isomorphism classes $a=[X]$ of objects $X\in \G$
where multiplication is induced by the tensor product, that is,
$[X][Y]=[X Y]$.

The construction $\G\mapsto M(\G)$  turns the category of monoidal
groupoids into a  fibred category over the category of monoids. To
determine its fiber over a monoid, we shall proceed as Schreier did
for extensions of a group.

We start by choosing a cleavage for $\G$ over $M(\G)$; that is, for
each $a\in M(\G)$, let us choose an object $X_a\in a$, and for any
other $X\in a$, we fix a morphism $\Gamma=\Gamma_X:X\to X_a$. In
particular,
 we take
 \begin{equation} \label{sel}\begin{array}{ccc}X_1=\I,&\hspace{0.4cm} &\Gamma_{X_a}=1_{X_a}:X_a\to X_a,\\
\Gamma_{\I X_a}=\boldsymbol{l}_{X_a}:\I X_a\to X_a,& &\Gamma_{X_a \I}=\boldsymbol{r}_{X_a}:X_a\I\to X_a.
\end{array}\end{equation}

Then, we have the following family of isotropy groups of the
groupoid $\G$ parameterized by the elements of $M(\G)$:
\begin{equation}\label{a}
\AAw(\G)=\big(\mathrm{Aut}_\G(X_a)\big)_{a\in M(\G)}.
\end{equation}

We also have the family of group homomorphisms
\begin{equation}\label{ab} \Theta(\G)=\big(\mathrm{Aut}_\G(X_b)\overset{a_*}\longrightarrow \mathrm{Aut}_\G(X_{ab})
\overset{b^*}\longleftarrow \mathrm{Aut}_\G(X_a)   \big)_{a,b\in M(\G)}\,,
\end{equation}
which, for any $a,b\in M$, carry automorphisms of $\G$, say
$f:X_b\to X_b$ and ${g:X_a\to X_a}$, to the automorphisms
$a_*(f):X_{ab}\to X_{ab}$ and $b^*(g):X_{ab}\to X_{ab}$,
respectively determined by the commutativity of the squares below.
 \begin{equation}\label{dab}\begin{array}{c}
  \xymatrix{X_a X_b\ar[d]_{\Gamma}\ar[r]^{1 f}&X_a X_b\ar[d]^{\Gamma}\\ X_{ab}\ar@{.>}[r]^{a_*(f)}&X_{ab}} \hspace{0.6cm}
  \xymatrix{X_a X_b\ar[d]_{\Gamma}\ar[r]^{g 1}&X_a X_b\ar[d]^{\Gamma}\\ X_{ab}\ar@{.>}[r]^{b^*(g)}&X_{ab}}
  \end{array}
\end{equation}

Furthermore, for any three elements  $a,b,c\in M(\G)$, there is a
unique
$$\lambda_{a,b,c}\in \mathrm{Aut}_\G(X_{abc})$$
making commutative the diagram
\begin{equation}\label{dlam}\begin{array}{c}
\xymatrix@C=40pt{
(X_a X_b) X_c \ar[d]_{\boldsymbol{a}}\ar[r]^-{\Gamma 1}&X_{ab} X_c\ar[r]^-{\Gamma}&X_{abc}\ar@{.>}[d]^{\lambda_{a,b,c}}
\\ X_a (X_b X_c)\ar[r]^-{1 \Gamma}&X_a X_{bc}\ar[r]^-{\Gamma}&X_{abc}.
}\end{array}
\end{equation}

Then, letting
\begin{equation}\label{abc}
\lambda(\G)=\big(\lambda_{a,b,c}\in \mathrm{Aut}_\G(X_{abc})\big)_{\!a,b,c\,\in M(\G)}\, ,
\end{equation}
we have:

\begin{proposition} For any monoidal groupoid
$\G=(\G,\otimes,\mathrm{I},\boldsymbol{a},\boldsymbol{l},\boldsymbol{r})$,
the associated quadruplet
\begin{equation}\label{assch}\Delta(\G)=\big(M(\G),\AAw(\G),\Theta(\G),\lambda(\G)\big),\end{equation} given by
$(\ref{m})$, $(\ref{a})$, $(\ref{ab})$, and $(\ref{abc})$, is a
Schreier system.
\end{proposition}
\begin{proof} In all the diagrams below, those inner regions labelled with
$(A)$ commute by the naturality of the associativity constraint,
those labelled with $(B)$ are commutative because the tensor product
$\otimes:\G\times \G\to\G$ is a functor, and the others commute by
the references therein.

For any  $a,b,c\in M(\G)$, $h\in \mathrm{Aut}_\G(X_{a})$, $g\in
\mathrm{Aut}_\G(X_{b})$, and $f\in\mathrm{Aut}_\G(X_{c})$, the
conditions in $(\ref{c11})$, $(\ref{c12})$, and $(\ref{c13})$,
follow, respectively, from the commutativity of the outside regions
in the following three diagrams in $\G$:
$$
\xymatrix@C=15pt{
X_{abc}\ar@/^2.3pc/[rrrrr]^{\lambda_{a,b,c}}\ar[d]_{(ab)_*(f)}
\ar@{}@<20pt>[d]|{(\ref{dab})}&X_{ab}X_c\ar[l]_(.4){\Gamma}\ar[d]|{1\,f}&(X_aX_b)X_c
\ar@{}@<16pt>[r]|{(\ref{dlam})}
\ar@{}[rd]|{(A)}\ar@{}[ld]|{(B)}\ar[d]|{(11) f}\ar[l]_-{\Gamma 1}
\ar[r]^{\boldsymbol{a}}&X_a(X_bX_c)\ar@{}[rd]|{(\ref{dab})}\ar[d]|{1(1 f)}
\ar[r]^-{1\Gamma}&X_a X_{bc} \ar[d]|{1 b_*{\!(f)}} \ar[r]^(.4){\Gamma}&
X_{abc}\ar[d]^{a_*b_*(f)} \ar@{}@<-20pt>[d]|{(\ref{dab})}\\X_{abc}\ar@/_2.3pc/[rrrrr]_{\lambda_{a,b,c}}
&X_{ab}X_c\ar[l]_{\Gamma}&(X_aX_b)X_c\ar@{}@<-16pt>[r]|{(\ref{dlam})}
\ar[l]_-{\Gamma 1}\ar[r]^{\boldsymbol{a}}&X_a(X_bX_c)\ar[r]^-{1\Gamma}&X_a X_{bc} \ar[r]^{\Gamma}
&X_{abc}
}
$$
$$
\xymatrix@C=15pt{
X_{abc}\ar@/^2.3pc/[rrrrr]^{\lambda_{a,b,c}}\ar[d]_{c^*a_*(g)}
\ar@{}@<20pt>[d]|{(\ref{dab})}&X_{ab}X_c\ar[l]_(.4){\Gamma}\ar[d]|{a_*(g)1}&(X_aX_b)X_c
\ar@{}@<16pt>[r]|{(\ref{dlam})}
\ar@{}[rd]|{(A)}\ar@{}[ld]|{(\ref{dab})}\ar[d]|{(1 g) 1}
\ar[l]_-{\Gamma
1}\ar[r]^{\boldsymbol{a}}&X_a(X_bX_c)\ar@{}[rd]|{(\ref{dab})}
\ar[d]|{1(g1)}\ar[r]^-{1\Gamma}&X_a X_{bc} \ar[d]|{1 c^*{\!(g)}} \ar[r]^(.4){\Gamma}&
X_{abc}\ar[d]^{a_*c^*(g)} \ar@{}@<-20pt>[d]|{(\ref{dab})}\\
X_{abc}\ar@/_2.3pc/[rrrrr]_{\lambda_{a,b,c}}&X_{ab}X_c\ar[l]_{\Gamma}&(X_aX_b)X_c
\ar@{}@<-16pt>[r]|{(\ref{dlam})}\ar[l]_-{\Gamma 1}
\ar[r]^{\boldsymbol{a}}&X_a(X_bX_c)\ar[r]^-{1\Gamma}&X_a X_{bc} \ar[r]^{\Gamma}&X_{abc}
}
$$
$$
\xymatrix@C=15pt{
X_{abc}\ar@/^2.3pc/[rrrrr]^{\lambda_{a,b,c}}\ar[d]_{c^*b^*(h)}\ar@{}@<20pt>[d]|{(\ref{dab})}&
X_{ab}X_c\ar[l]_(.4){\Gamma}\ar[d]|{b^*\!(h)1}&(X_aX_b)X_c \ar@{}@<16pt>[r]|{(\ref{dlam})}
\ar@{}[rd]|{(A)}\ar@{}[ld]|{(\ref{dab})}\ar[d]|{(h 1)1}\ar[l]_-{\Gamma 1}\ar[r]^{\boldsymbol{a}}&
X_a(X_bX_c)\ar@{}[rd]|{(B)}\ar[d]|{h(11)}\ar[r]^-{1\Gamma}&X_a X_{bc}
 \ar[d]|{h1} \ar[r]^(.4){\Gamma}&X_{abc}\ar[d]^{(bc)^*(h)} \ar@{}@<-20pt>[d]|{(\ref{dab})}\\
X_{abc}\ar@/_2.3pc/[rrrrr]_{\lambda_{a,b,c}}&X_{ab}X_c\ar[l]_{\Gamma}&(X_aX_b)X_c
\ar@{}@<-16pt>[r]|{(\ref{dlam})}
\ar[l]_-{\Gamma 1}\ar[r]^{\boldsymbol{a}}&X_a(X_bX_c)\ar[r]^-{1\Gamma}&X_a X_{bc} \ar[r]^{\Gamma}&X_{abc}
}
$$

The 3-cocycle condition $(\ref{c2})$, for any $a,b,c, d\in M(\G)$,
follows from the commutativity of the following diagram:
$$ \xymatrix@C=-28.5pt@R=18pt{
X_{abcd}\ar@/_2.2pc/@<-5pt>[dddddddd]|(.35){d^*(\lambda_{a,b,c})}\ar[rrrr]^{\lambda_{ab,c,d}}&&&&
X_{abcd}\ar[rrrr]^{\lambda_{a,b,cd}}&&&&X_{abcd}\ar@{<-}@/^2.2pc/@<5pt>[dddddddd]|(.35){a_*(\lambda_{b,c,d})}
\\
X_{abc}X_d\ar@{}[rrrr]|{(\ref{dlam})}\ar@/_1pc/[dddddd]|(.6){\lambda_{a,b,c}1}
\ar@{}@<-9pt>[dddddd]_(.4){(\ref{dab})}\ar[u]_{\Gamma}&
&&& X_{ab}X_{cd}\ar[u]^{\Gamma}\ar@{}[rrrr]|{(\ref{dlam})}
&&&&X_aX_{bcd}\ar[u]^{\Gamma}
\ar@{<-}@/^1pc/[dddddd]|(.6){1\lambda_{b,c,d}}\ar@<9pt>@{}[dddddd]^(.4){(\ref{dab})}
 \\
&(X_{ab}X_c)X_d \ar@{}[ddd]|{(\ref{dlam})}\ar[lu]_{\Gamma
1}\ar[rr]^{\boldsymbol{a}}&&X_{ab}(X_cX_d)
\ar[ru]^{1\Gamma}&&(X_aX_b)X_{cd}\ar[lu]_{\Gamma
1}\ar[rr]^{\boldsymbol{a}}
&&X_a(X_bX_{cd})\ar[ru]^{1\Gamma}\ar@{}[ddd]|{(\ref{dlam})}&\\
 & & & &(X_aX_b)(X_cX_d)\ar[lu]\ar@<-2pt>@{}[lu]^(.6){\Gamma(
11)}\ar[ru]\ar@<2pt>@{}[ru]_(.6){(11) \Gamma}
\ar[rrd]^{\boldsymbol{a}} & & & & \\
&&((X_aX_b)X_c)X_d\ar@{}[uu]_{(A)}\ar[d]^{\boldsymbol{a}1}
\ar[rru]^{\boldsymbol{a}}\ar[luu]\ar@<-2pt>@{}[luu]^(.6){(\Gamma 1)
1}&& &&X_a(X_b(X_cX_d))\ar@{}[uu]^{(A)}\ar[ruu]\ar@<2pt>@{}[ruu]_(.6){1(1\Gamma)}
\ar@{<-}[d]^{1\boldsymbol{a}}&&\\
&&(X_a(X_bX_c))X_d\ar[rrrr]^{\boldsymbol{a}}\ar[ld]_{(1\Gamma)1}&&&&X_a((X_bX_c)X_d))
\ar[rd]^{1(\Gamma 1)}&&\\
&(X_aX_{bc})X_d\ar[rrrrrr]^{\boldsymbol{a}}\ar[ld]^{\Gamma
1}&&&&&&X_{a}(X_{bc}X_{d})\ar[rd]_{1\Gamma}
\\
X_{abc}X_{d}\ar[d]^{\Gamma}&&&&&&&&X_aX_{bcd}\ar[d]_{\Gamma}\\
X_{abcd}\ar[rrrrrrrr]^{\lambda_{a,bc,d}}&&&
&\ar@{}[uuuuuu]|(.2){(\ref{dlam})}|(.4){(A)}|(.6){(\ref{mac1})}|(.99){(B)}&&&&X_{abcd}.
} $$

Since, for any  $a,b\in M(\G)$, $g\in \mathrm{Aut}_\G(X_{a})$, and
$f\in\mathrm{Aut}_\G(X_{b})$, we have the commutative diagram
$$
\xymatrix{
X_{ab}\ar@{}@<16pt>[rrr]|{(\ref{dab})}\ar@/^2.2pc/[rrr]^{a_*(f)}\ar[d]_{b^*(g)}
\ar@{}@<20pt>[d]|{(\ref{dab})}&X_a X_b\ar@{}[rd]|{(B)}\ar[d]_{g 1}\ar[r]^{1 f}\ar[l]_{\Gamma}&
X_a X_b\ar[r]^{\Gamma}\ar[d]^{g 1}&X_{ab}\ar[d]^{b^*(g)}\ar@{}@<-20pt>[d]|{(\ref{dab})}\\
X_{ab}\ar@/_2.2pc/[rrr]_{a_*(f)}\ar@{}@<-16pt>[rrr]|{(\ref{dab})}&X_a X_b\ar[r]_{1 f}\ar[l]_{\Gamma}&
X_a X_b\ar[r]^{\Gamma}& X_{ab},
}
$$
it follows that the homomorphisms $a_*$ and $b^*$ in (\ref{ab}) are
centralizing, that is, the condition in $(\ref{c3})$ holds. Moreover,
when $a=1$ or $b=1$, the naturality of the unit constraints gives
the commutativity of the squares
$$
\xymatrix@R=15pt{\I X_b\ar[d]_{\Gamma=\boldsymbol{l}}\ar[r]^{1 f}&\I X_b\ar[d]^{\Gamma=\boldsymbol{l}}\\ X_{b}\ar[r]^{f}&X_{b},} \hspace{0.6cm}
  \xymatrix@R=15pt{X_a \I\ar[d]_{\Gamma=\boldsymbol{r}}\ar[r]^{g1}&X_a \I\ar[d]^{\Gamma=\boldsymbol{r}}\\ X_{a}\ar[r]^{g}&X_{a},}
$$
whence $ 1_*(f)=f$ and $1^*(g)=g$. That is, the normalization
conditions in $(\ref{c4})$ hold. Furthermore, recalling the
selections (\ref{sel}), it is plain to see that  the normalization
conditions in $(\ref{c5})$ a are direct consequence of the coherence
triangles in $(\ref{mac1})$ and $(\ref{tm2tr})$.  This completes the
proof.\qed
\end{proof}

The Schreier system in $(\ref{assch})$, associated to a monoidal
groupoid, depends on the selection of the cleavage  made for its
construction. However, as we shall prove, different choices produce
{\em equivalent} Schreier systems.

We next explain how each Schreier system gives rise, by the
Grothendieck construction (cf. \cite[1.3]{car94}), to a monoidal
groupoid.

\subsection{The monoidal groupoid defined by a Schreier system} Let $\mathcal{S}=(M,\AAw,\Theta,\lambda)$ be
any given Schreier system. Then, a monoidal groupoid
\begin{equation}\label{gc} \Sigma(\mathcal{S})=(\Sigma(\Sw),\otimes,\mathrm{I},\boldsymbol{a},\boldsymbol{l},
\boldsymbol{r})
\end{equation}
is defined as follows: an object of $\Sigma(\Sw)$ is an element $a\in
M$. If $a\neq b$ are different elements of the monoid $M$, then
there are no morphisms in $\Sigma(\Sw)$ between them,  whereas if
$a=b$, then a morphism $f:a\to a$ is an element $f$ of the group
$A_a$, that is,
$$
\Sigma(\Sw)(a,b)=\left\{
\begin{array}{lll}\emptyset&\text{if} &a\neq b,\\[4pt]
A_a&\text{if}&a=b.\end{array}     \right.
$$

The composition of morphisms is given by the group operation of
$A_a$, that is,
$$
(a\overset{f}\to a)\circ (a\overset{f'}\to a)=(a\overset{f\circ f' }\longrightarrow a).
$$

The tensor product $\otimes:\Sigma(\Sw)\times \Sigma(\Sw)\to
\Sigma(\Sw)$ is defined by
$$
(a\overset{g}\to a)\otimes (b\overset{f}\to b)=\xymatrix@C=45pt{(ab\ar[r]^{a_*(f)\circ b^*(g)}&ab)},
$$
which is a functor thanks to the centralizing condition (\ref{c3}).
In effect, we have

$$
(a\overset{1}\to a)\otimes (b\overset{1}\to b)=\xymatrix@C=45pt{(ab\ar[r]^{a_*(1)\circ b^*(1)}&ab)}
= ab\overset{1}\to ab.
$$
and,  for any $g,g':a\to a$ and $f,f':b\to b$, we have
$$
\begin{array}{lll}(g\circ g')\otimes (f\circ f')&=&a_*(f\circ f')\circ b^*(g\circ g')\ = \ a_*(f)\circ a_*(f')
\circ b^*(g)\circ b^*(g')\\&\overset{(\ref{c3})}=&a_*(f)\circ b^*(g)
\circ a_*(f')\circ b^*(g')\ =\ (g\otimes f)\circ (g'\otimes f').
\end{array}
$$

The associativity isomorphisms are
$$
\lambda_{a,b,c}:(ab)c\to a(bc).
$$
These are natural thanks to conditions (\ref{c11}), (\ref{c12}), and
(\ref{c13}). In effect, for any $h:a\to a$, $g:b\to b$, and $f:c\to
c$,
$$
\begin{array}{lcl} \lambda_{a,b,c}\circ ((h\otimes g)\otimes f)&=&\lambda_{a,b,c}\circ
((a_*(g)\circ b^*(h))\otimes f)\\[4pt] &=& \lambda_{a,b,c} \circ (ab)_*(f)\circ c^*(a_*(g)\circ b^*(h))
\\[4pt] &=&  \lambda_{a,b,c} \circ (ab)_*(f)\circ c^*(a_*(g))\circ c^*(b^*(h))\\ & \overset{(\ref{c11})}=&
a_*(b_*(f))\circ \lambda_{a,b,c}\circ c^*(a_*(g))\circ c^*(b^*(h))\\[4pt]
 & \overset{(\ref{c12})}=&
a_*(b_*(f))\circ a_*(c^*(g))\circ \lambda_{a,b,c}\circ c^*(b^*(h))
\\[4pt]
 & \overset{(\ref{c13})}=&
a_*(b_*(f))\circ a_*(c^*(g))\circ (bc)^*(h)\circ \lambda_{a,b,c}\\[4pt]
& =&
a_*(b_*(f)\circ c^*(g))\circ (bc)^*(h)\circ \lambda_{a,b,c}\\[4pt]
&=& (h\otimes  (b_*(f)\circ c^*(g)))\circ \lambda_{a,b,c}\\[4pt]&=&(h\otimes(g\otimes f))  \circ \lambda_{a,b,c}.
\end{array}
$$
The pentagon coherence condition in (\ref{mac1}) just says that, for
any $a,b,c,d\in M$, the diagram
$$
\xymatrix{((ab)c)d
\ar[r]^{\lambda_{ab,c,d}}\ar[d]_{ d^*(\lambda_{a,b,c})}&(ab)(cd) \ar[r]^{\lambda_{a,b,cd}}& a(b(cd))
\\ (a(bc))d\ar[rr]^{\lambda_{a,bc,d}}&&a((bc)d)\ar[u]_{a_*(\lambda_{b,c,d})}}
$$
must be commutative, which holds because of the 3-cocycle condition
(\ref{c2}).

 The unit object is $\I=1$, the unit element
of the monoid $M$, and the unit constraints are both identities,
that is, for any $a\in M$,
$$\boldsymbol{l}_a=1=\boldsymbol{r}_a:a\to a.$$
These are natural due to the equalities in $(\ref{c4})$. In effect,
for any $f:a\to a$, we have $$\begin{array}{l} \boldsymbol{l}_a\circ
(1\otimes f)= 1\otimes f= 1_*(f)\circ
a^*(1)\overset{(\ref{c4})}=f\circ
1=f=f\circ \boldsymbol{l}_a,\\[5pt]
\boldsymbol{r}_a\circ (f\otimes 1)= f\otimes 1= a_*(1)\circ
1^*(f)\overset{(\ref{c4})}= 1\circ f=f= f \circ \boldsymbol{r}_a.
\end{array}
$$
 The coherence triangle for the unit in
$(\ref{mac1})$ commutes owing to the normalization condition
$\lambda_{a,1,c}=1$ in (\ref{c5}).\qed

As we will show, both constructions $\Sw\mapsto \Sigma(\Sw)$, as
above, and $\G\mapsto \Delta(\G)$, as  in (\ref{assch}), are
suitable for expressing the strong relationship between Schreier
systems and monoidal groupoids. We need the notions of {\em
morphisms} between Schreier systems and their {\em deformations},
which we establish below.

\subsection{The 2-category of Schreier systems} The Schreier systems introduced in Definition \ref{defSch} (or
non-abelian 3-cocycles of monoids) are the objects of a 2-category
in which all 2-cells are invertible, denoted by
$$\mathbf{Z^3_{\mathrm{n\text{-}ab}}Mnd},$$
whose cells and their compositions are defined as follows:

\subsubsection{Morphisms of Schreier systems}\label{morscsys}

 If $\mathcal{S}= (M,\AAw,\Theta,\lambda)$ and
$\mathcal{S'}= (M',\AAw',\Theta',\lambda')$ are two Schreier
systems, then a {\em morphism}
$$\wp=(p,\mathbbm{q}\,,\varphi):\Sw\to\Sw'$$ consists of the
following data:
\begin{itemize}
\item[$\bullet$] a monoid homomorphism $p:M\to M'$,
\item[$\bullet$] a family of group homomorphisms $\mathbbm{q}=\xymatrix@C=15pt{\big(A_a\ar[r]^-{q_a}&
A_{p(a)}^\prime\big)_{\!a\in M}}$,
\item[$\bullet$] a family of elements $\varphi=\big(\varphi_{a,b}\in A'_{p(ab)}\big)_{\!a,b\in M}$,
\end{itemize}
satisfying the following three conditions:

$\bullet$ For any $a,b\in M$, $g\in A_a$, and $f\in A_b$,
\begin{equation}\label{cms1}\begin{array}{l}
\varphi_{a,b}\circ p(a)_*(q_b(f)) \circ \varphi_{a,b}^{-1}= q_{ab}(a_*(f)) ,\\[5pt]
\varphi_{a,b}\circ p(b)^*(q_a(g)) \circ \varphi_{a,b}^{-1}= q_{ab}(b^*(g)).
\end{array}
\end{equation}

$\bullet$ For any $a,b,c\in M$,
\begin{equation}\label{cms3}
q_{abc}(\lambda_{a,b,c})\circ \varphi_{ab,c}\circ p(c)^*(\varphi_{a,b})=\varphi_{a,bc}\circ p(a)_*(\varphi_{b,c})\circ \lambda'_{p(a),p(b),p(c)}.
\end{equation}

$\bullet$ $\varphi$ is normalized, that is,  \begin{equation}\label{cms4}
\varphi_{1,1}=1.\end{equation}

 Observe that, taking $b=c=1$ in the above equality
$(\ref{cms3})$, we deduce that, for any $a\in M$,  $\varphi_{a,1}
\circ \varphi_{a,1}=\varphi_{a,1}\circ
p(a)_*(\varphi_{1,1})=\varphi_{a,1}$ in the group $A'_{p(a)}$,
whence $\varphi_{a,1}=1$. Similarly, $\varphi_{1,a}=1$.

\subsubsection{Deformations}\label{defmorSchr}
Let $\wp=(p,\mathbbm{q}\,,\varphi):\Sw\to\Sw'$ and
$\bar{\wp}=(\bar{p},\bar{\mathbbm{q}},\bar{\varphi}):\Sw\to\Sw'$ be
morphisms between Schreier systems $\Sw=(M,\AAw,\Theta,\lambda)$ and
$\Sw'=(M',\AAw',\Theta',\lambda')$.

If $p\neq \bar{p}$ are different homomorphisms, then there is no deformation between
$\wp$ and $\bar{\wp}$ in $\mathbf{Z^3_{\mathrm{n\text{-}ab}}Mnd}$.

If $p=\bar{p}$, then  a {\em deformation}
$$\xymatrix @C=22pt{\Sw
\ar@/^0.7pc/[rr]^{\wp=(p,\mathbbm{q}\,,\varphi)}
 \ar@/_0.7pc/[rr]_{\bar{\wp}=(p,\bar{\mathbbm{q}},\bar{\varphi})}\ar@{}[rr]|{\Downarrow\,\delta} &
&\Sw'}$$ is a family of elements $ \delta=\big(\delta_a\in
A'_{p(a)}\big)_{a\in M} $ satisfying the following two conditions:

$\bullet$ For any $a\in M$ and $f\in A_a$,
\begin{equation}\label{cbms1} \delta_a^{-1}\circ \bar{q}_a(f)\circ \delta_a= q_a(f).\end{equation}

$\bullet$ For any $a,b\in M$,
\begin{equation}\label{cbms2} \delta_{ab}\circ \varphi_{a,b}=\bar{\varphi}_{a,b}\circ p(a)_*(\delta_b)\circ p(b)^*(\delta_a).
\end{equation}

Observe that, taking $a=b=1$ in the above equality $(\ref{cbms2})$,
we deduce that $\delta_1=\delta_1\circ \delta_1$ in the group
$A'_1$, whence $\delta_1=1$.

\subsubsection{Vertical composition of deformations}
For any Schreier systems $\mathcal{S}= (M,\AAw,\Theta,\lambda)$ and
$\mathcal{S'}= (M',\AAw',\Theta',\lambda')$, the vertical composition
in the 2-category $\mathbf{Z^3_{\mathrm{n\text{-}ab}}Mnd}$ of
deformations
\begin{equation}\label{ddforv}\xymatrix @C=22pt{\Sw
\ar@/^1.5pc/[rr]^{\wp=(p,\mathbbm{q}\,,\varphi)}\ar[rr]|{(p,\bar{\mathbbm{q}},\bar{\varphi})}
 \ar@/_1.5pc/[rr]_{\bar{\bar{\wp}}=(p,\bar{\bar{\mathbbm{q}}},\bar{\bar{\varphi}})}
 \ar@<10pt>@{}[rr]|{\Downarrow\,\delta}\ar@<-10pt>@{}[rr]|{\Downarrow\,\bar{\delta}} &
&\Sw'}\end{equation} is the deformation $\bar{\delta}\circ \delta:\wp\Rightarrow\bar{\bar{\wp}}$
 obtained by pointwise multiplication, that is,
\begin{equation}\label{fddforv}
\bar{\delta}\circ\delta=\big(\bar{\delta}_a\circ\delta_a\in A'_{p(a)}\big)_{a\in M}.
\end{equation}
The identity deformation on each morphism $\wp:\Sw\to\Sw'$ is
$$1_\wp=\big(1\in A'_{p(a)}\big)_{a\in M}:\wp\Rightarrow\wp.$$

Every deformation $\delta:\wp\Rightarrow\wp'$ is invertible, with
inverse $\delta^{-1}=(\delta_a^{-1}\in A'_{p(a)})_{a\in M}$.
Therefore, the
hom-categories $\mathbf{Z^3_{\mathrm{n\text{-}ab}}Mnd}(\Sw,\Sw')$
are groupoids.

\subsubsection{Horizontal composition of morphisms}
 For  Schreier
systems $\Sw=(M,\AAw,\Theta,\lambda)$,
$\Sw'=(M',\AAw',\Theta',\lambda')$,
$\Sw''=(M'',\AAw'',\Theta'',\lambda'')$, the horizontal composition
of two morphisms
$$
\xymatrix@C=55pt{\Sw\ar[r]^{\wp=(p,\mathbbm{q}\,,\varphi)}&\Sw'\ar[r]^{\wp'=(p',\mathbbm{q}',\varphi')}&\Sw''}
$$
is the morphism
\begin{equation}\label{cp'p}\wp'\wp=(p'p,\mathbbm{q}'\mathbbm{q},\varphi\varphi'):\Sw\to
\Sw'',\end{equation} where $p'p:M\to M''$ is the composite of $p$ and $p'$, and
$$\begin{array}{l}
\mathbbm{q}'\mathbbm{q}=\big(q'_{p(a)}q_a:A_a\to A''_{p'\!p\,(a)}\big)_{a\in M},\\[6pt]
\varphi\varphi'=\big(q'_{p(ab)}(\varphi_{a,b})\circ
\varphi'_{p(a),p(b)}\in A''_{p'\!p(ab)}\big)_{a,b\in M}.\end{array}
$$

The identity morphism on a Schreier system
$\Sw=(M,\AAw,\Theta,\lambda)$ is
\begin{equation}\label{cid}1_\Sw=(1_M,\mathbbm{1}_{\AAw},1):\Sw\to\Sw,\end{equation} where $1_M$ is the identity
map on $M$, $\mathbbm{1}_{\AAw}=\big(1_{A_a}\big)_{a\in M}$,
and $1=\big(1\in A_{ab}\big)_{a,b\in M}$.

\subsubsection{Horizontal composition of deformations} The
horizontal composition of deformations
\begin{equation}\label{chd}\xymatrix @C=22pt {\Sw  \ar@/^0.8pc/[rr]^{\wp=(p,\mathbbm{q}\,,\varphi)}
 \ar@/_0.8pc/[rr]_{\bar{\wp}=(p,\bar{\mathbbm{q}},\bar{\varphi})}\ar@{}[rr]|{\Downarrow\,\delta} &
&\Sw' \ar@/^0.8pc/[rr]^{\wp'=(p',\mathbbm{q}',\varphi')}
 \ar@/_0.8pc/[rr]_{\bar{\wp}'=(p',\bar{\mathbbm{q}}',\bar{\varphi}')}\ar@{}[rr]|{\Downarrow\,\delta'} &
&\Sw''}\end{equation} is the deformation $\delta'\delta:\wp'\wp\Rightarrow\bar{\wp}'\bar{\wp}$ defined by
\begin{equation}\label{cdeltas}\delta'\delta=\big(\delta'_{p(a)}\circ
q'_{p(a)}(\delta_a)\in A''_{p'\!p(a)}\big)_{a\in M}.\end{equation}

For later use, we prove here the lemma below.
\begin{lemma}\label{iseqin} For any morphism
$(p,\mathbbm{q}\,,\varphi):(M,\AAw,\Theta,\lambda)\to
(M',\AAw',\Theta',\lambda')$ in the $2$-category
$\mathbf{Z^3_{\mathrm{n\text{-}ab}}Mnd}$, the following statements
are equivalent:

$(i)$ $(p,\mathbbm{q}\,,\varphi)$ is an isomorphism.

$(ii)$ $(p,\mathbbm{q}\,,\varphi)$ is an equivalence.

$(iii)$ The homomorphisms $p:M\to M'$ and $q_a:A_a\to A'_{p(a)}$,
$a\in M$, are all isomorphisms.
\end{lemma}
\begin{proof} $(i)\Rightarrow (ii)$ is obvious.

$(ii)\Rightarrow (iii)$. First observe that, for any Schreier system
$\Sw=(M,\AAw,\Theta,\lambda)$, a morphism $\wp:\Sw\to \Sw$ with a
deformation $\delta:\wp\Rightarrow 1_\Sw$ is necessarily of the form
$\wp=(1_M,\mathbbm{q}(\delta),\varphi(\delta))$, for some family
$\delta=(\delta_a\in A_a)_{a\in M}$, with $\delta_1=1$, where
$\mathbbm{q}(\delta)=(q(\delta)_a:A_a\to A_a)_{a\in M}$ consists of
the inner automorphisms given by $q(\delta)_a (f)=\delta_a^{-1}
\circ f \circ \delta_a$, and
$\varphi(\delta)=(\varphi(\delta)_{a,b}\in A_{ab})_{a,b\in M}$
consists of the elements obtained by the formula
$\varphi(\delta)_{a,b}=\delta_{ab}^{-1}\circ a_*(\delta_b)\circ
b^*(\delta_a)$.

Then, the existence of a morphism $(p',\mathbbm{q}',\varphi'):\Sw'\to
\Sw$, where $\Sw$ is as above and $\Sw'=(M',\AAw',\Theta',\lambda')$,
with deformations
$\delta:(p',\mathbbm{q}',\varphi')(p,\mathbbm{q},\varphi)\Rightarrow
1_\Sw$ and ${\delta':(p,\mathbbm{q},\varphi)(p',\mathbbm{q}',\varphi')
\Rightarrow 1_{\Sw'}}$, implies that $ p'p=1_M,\ pp'=1_{M'}, $ so $p$
is an isomorphism, and also that
$$
q'_{p(a)}q_a=q(\delta)_a, \ \  q_aq'_{p(a)}=q(\delta')_{p(a)},$$ for
all $a\in M$. Hence  $q_a$ and $q'_{p(a)}$ are both isomorphisms
since $q(\delta)_a$ and  $q(\delta')_{p(a)}$ are automorphisms.

$(iii)\Rightarrow (i)$. The inverse
$(p,\mathbbm{q}\,,\varphi)^{-1}=(p',\mathbbm{q}',\varphi')$ is given
by taking $$p'=p^{-1}, \ \
\mathbbm{q}'=\big(q_{p'(a')}^{-1}\big)_{a'\in M'},\ \
\varphi'=\big(q'_{a'b'}(\varphi_{p'(a'),p'(b')}^{-1})\big)_{a',b'\in
M'}.$$ \qed
\end{proof}
\subsection{The classifying biequivalence}
The following theorem, where it is stated that the 2-categories of
Schreier systems and monoidal groupoids  are biequivalent,  is the
main result of this section.

\begin{theorem}[Classification of monoidal groupoids]\label{mt1}
The assignment $\Sw\mapsto \Sigma(\Sw)$, given by the monoidal
groupoid construction $(\ref{gc})$, is the function on objects of a
$2$-functor
\begin{equation}\label{tsigma}\xymatrix@C=16pt{\Sigma:\mathbf{Z^3_{\mathrm{n\text{-}ab}}Mnd}
\ar[r]\ar@{}@<-2.5pt>[r]^(.51){\approx}&
\mathbf{MonGpd},}\end{equation} which establishes a biequivalence
between the $2$-category of Schreier systems and the $2$-category of
monoidal groupoids. More precisely (cf. \cite[p. 570]{street}),  for
any two Schreier systems $\Sw$ and $\Sw'$, the functor
\begin{equation}\label{lsigma}\xymatrix@C=16pt{\Sigma:
\mathbf{Z^3_{\mathrm{n\text{-}ab}}Mnd}(\Sw,\Sw')\ar[r]\ar@{}@<-2.5pt>[r]^(.46){\thicksim}&
\mathbf{MonGpd}(\Sigma(\Sw),\Sigma(\Sw'))}\end{equation} is an
equivalence of groupoids, and for any monoidal groupoid $\G$, there
exists  a monoidal equivalence
\begin{equation}\label{psigma}
\xymatrix@C=16pt{J_\G:\Sigma(\Delta(\G))\ar[r]\ar@{}@<-2.5pt>[r]^(.65){\thicksim}& \G,}
\end{equation}
where $\Delta(\G)$ is the Schreier system $(\ref{assch})$ associated
to $\G$.
\end{theorem}
\begin{proof}
We have already described $\Sigma$ on objects of the 2-category
$\mathbf{Z^3_{\mathrm{n\text{-}ab}}Mnd}$; its effect on morphisms
and deformations is as follows:
\subsubsection{$\Sigma$ on morphisms} Let $\mathcal{S}=
(M,\AAw,\Theta,\lambda)$, $\mathcal{S'}= (M',\AAw',\Theta',\lambda')$
be Schreier systems. Then, the 2-functor $\Sigma$ carries any
morphism $\wp=(p,\mathbbm{q}\,,\varphi):\Sw\to\Sw'$ to the strictly
unitary monoidal functor $ \Sigma(\wp):\Sigma(\Sw)\to \Sigma(\Sw') $
given  by
\begin{equation}\label{defsig}
\big(a\overset{f}\to a\big)\ \mapsto \big(p(a)\overset{q_a(f)}\longrightarrow p(a)\big),
\end{equation}
and whose structure isomorphisms are
\begin{equation}\label{pab}
\varphi_{a,b}:p(a)p(b)\to p(ab),
\end{equation}
which are well defined since $p(a)p(b)=p(ab)$ and $\varphi_{a,b}\in
A'_{p(ab)}$, for any $a,b\in M$.

Since the maps $q_a:A_a\to A'_{p(a)}$ are homomorphisms, it follows
that $\Sigma(\wp)$ is a functor. Furthermore, the isomorphisms
$(\ref{pab})$ are natural since, for any morphisms $f:b\to b$ and
$g:a\to a$ in $\Sigma(\Sw)$, the squares in $\Sigma(\Sw')$
\begin{equation}\label{tsqus}\begin{array}{c}
\xymatrix{p(a)p(b)\ar[r]^{\varphi_{a,b}}\ar[d]_{p(a)_*(q_b(f))}&p(ab)\ar[d]^{q_{ab}(a_*(f))}\\
p(a)p(b)\ar[r]^{\varphi_{a,b}}&p(ab)}\hspace{0.3cm}
\xymatrix{p(a)p(b)\ar[r]^{\varphi_{a,b}}\ar[d]_{p(b)^*(q_a(g))}&p(ab)\ar[d]^{q_{ab}(b^*(f))}\\
p(a)p(b)\ar[r]^{\varphi_{a,b}}&p(ab)}\end{array}
\end{equation}
commute owing to condition $(\ref{cms1})$. The coherence condition
$(\ref{ecf})$ for $\Sigma(\wp)$ just says that the diagrams
\begin{equation}\label{conefc}\begin{array}{c}
\xymatrix@C=45pt{(p(a)p(b))p(c)\ar[d]_{\lambda'_{p(a),p(b),p(c)}}\ar[r]^-{p(c)^*(\varphi_{a,b})}&
p(ab)p(c)\ar[r]^{\varphi_{ab,c}}
&p((ab)c)\ar[d]^{q_{abc}(\lambda_{a,b,c})}\\p(a)(p(b)p(c))\ar[r]^-{p(a)_*(\varphi_{b,c})}&p(a)p(bc)
\ar[r]^{\varphi_{a,bc}}&p(a(bc))
}\end{array}
\end{equation}
must commute, which follows from $(\ref{cms3})$. Finally, conditions
$(\ref{ecf2})$ are both a consequence of the normality condition
$(\ref{cms4})$ of $\varphi$, that is, of the equalities
${\varphi_{a,1}=1=\varphi_{1,b}}$.

For $\wp'=(p',\mathbbm{q}',\varphi'):\Sw'\to\Sw''$ another Schreier
system morphism, the composite monoidal functor
$\Sigma(\wp')\Sigma(\wp):\Sw\to \Sw''$ is given by
$$
\Sigma(\wp')\Sigma(\wp)\big(a\overset{f}\to a\big)=\Sigma(\wp')
\big(p(a)\overset{q_a(f)}\longrightarrow p(a)\big)=
\big(\xymatrix@C=40pt{p'p(a)\ar[r]^{q'_{p(a)}(q_a(f))}&p'p(a)}),
$$
together with the structure isomorphisms obtained by composing in
$\Sigma(\Sw'')$
$$
\xymatrix@C=40pt{p'p(a)p'p(b)\ar[r]^{\varphi'_{p(a),p(b)}}&p'(p(a)p(b))\ar[r]^-{q'_{p(ab)}(\varphi_{a,b})}&p'p(ab)}.
$$
Hence, taking into account the definition of the composition
$\wp'\wp$ in $(\ref{cp'p})$ and the definition of $\Sigma$, simple
comparison shows that $\Sigma(\wp')\Sigma(\wp)=\Sigma(\wp'\wp)$.
Moreover, it is straightforward to see that  $\Sigma$ carries
identity morphisms on Schreier systems
$1_\Sw=(1_M,\mathbbm{1}_{\AAw},1)$, see $(\ref{cid})$, to
identity monoidal functors; that is, $\Sigma(1_\Sw)=1_{\Sigma(\Sw)}$
for any Schreier system $\Sw$. Therefore,
$\Sigma:\mathbf{Z^3_{\mathrm{n\text{-}ab}}Mnd}\to \mathbf{MonGpd}$
is indeed a functor.

\subsubsection{$\Sigma$ on deformations} Given Schreier systems $\Sw$ and
$\Sw'$ as above, any deformation
$$\xymatrix @C=22pt{\Sw
\ar@/^0.7pc/[rr]^{\wp=(p,\mathbbm{q}\,,\varphi)}
 \ar@/_0.7pc/[rr]_{\bar{\wp}=(p,\bar{\mathbbm{q}},\bar{\varphi})}\ar@{}[rr]|{\Downarrow\,\delta} &
&\Sw'}$$ is mapped by the 2-functor $\Sigma$ to the isomorphism of
monoidal functors
$$\xymatrix @C=22pt{\Sigma(\Sw)
\ar@/^0.8pc/[rr]^{\Sigma(\wp)}
 \ar@/_0.8pc/[rr]_{\Sigma(\bar{\wp})}\ar@{}[rr]|{\Downarrow\,\Sigma(\delta)} &
&\Sigma(\Sw')}$$ simply defined by the family of isomorphisms in $\Sigma(\Sw')$
\begin{equation}\label{defsigdel}
\Sigma(\delta)_a=\delta_a:p(a)\to p(a),\hspace{0.4cm}a\in M,
\end{equation}
which are natural thanks to condition $(\ref{cbms1})$. Moreover, so
defined, $\Sigma(\delta):\Sigma(\wp)\Rightarrow \Sigma(\bar{\wp})$
is monoidal, that is, conditions $(\ref{cfnt})$ hold, owing to
$(\ref{cbms2})$ and the equality $\delta_1=1\in A'_1$.

For any two vertically composable deformations
$\delta:\wp\Rightarrow\bar{\wp}$ and
$\bar{\delta}:\bar{\wp}\Rightarrow\bar{\bar{\wp}}$,  as in
$(\ref{ddforv})$, the equality $\Sigma(\bar{\delta}\circ
\delta)=\Sigma(\bar{\delta})\circ\Sigma(\delta)$ is easily verified
from $(\ref{fddforv})$ and $(\ref{pspps})$, as well as the equality
$\Sigma(1_\wp)=1_{\Sigma(\wp)}$, for any morphism $\wp:\Sw\to\Sw'$.
Hence, $(\ref{lsigma})$ is a functor.

Furthermore, for
$$\xymatrix @C=22pt {\Sw  \ar@/^0.8pc/[rr]^{\wp=(p,\mathbbm{q}\,,\varphi)}
 \ar@/_0.8pc/[rr]_{\bar{\wp}=(p,\bar{\mathbbm{q}},\bar{\varphi})}\ar@{}[rr]|{\Downarrow\,\delta} &
&\Sw' \ar@/^0.8pc/[rr]^{\wp'=(p',\mathbbm{q}',\varphi')}
 \ar@/_0.8pc/[rr]_{\bar{\wp}'=(p',\bar{\mathbbm{q}}',\bar{\varphi}')}\ar@{}[rr]|{\Downarrow\,\delta'} &
&\Sw''}$$
any two horizontally composable deformations as in $(\ref{chd})$, we
have the equality
$\Sigma(\delta'\delta)=\Sigma(\delta')\,\Sigma(\delta)$, since, for
any $a\in M$,
$$\begin{array}{lcl}
(\Sigma(\delta')\,\Sigma(\delta))_a&\overset{(\ref{psps})}=&\Sigma(\delta')_{\Sigma(\bar{\wp})(a)}\circ \Sigma(\wp')(\Sigma(\delta)_a)\\[5pt] &\overset{(\ref{defsig},\ref{defsigdel})}=& \delta'_{p(a)}\circ q'_{p(a)}(\delta_a) \overset{(\ref{cdeltas})}= (\delta'\delta)_a\overset{(\ref{defsigdel})}=\Sigma(\delta'\delta)_a.\end{array}
$$

The above confirms that $(\ref{tsigma})$,
$\Sigma:\mathbf{Z^3_{\mathrm{n\text{-}ab}}Mnd}\to \mathbf{MonGpd}$,
is actually a 2-functor.

\subsubsection{The functor $\Sigma$ in $(\ref{lsigma})$ is full and faithful} For any two
Schreier systems $\mathcal{S}= (M,\AAw,\Theta,\lambda)$ and
$\mathcal{S'}= (M',\AAw',\Theta',\lambda')$,  the functor $\Sigma:
\mathbf{Z^3_{\mathrm{n\text{-}ab}}Mnd}(\Sw,\Sw')\to
\mathbf{MonGpd}(\Sigma(\Sw),\Sigma(\Sw'))$ is plainly recognized to be
faithful, due to $(\ref{defsigdel})$. To prove that it is full, let
$\delta:\Sigma(\wp)\Rightarrow\Sigma(\bar{\wp})$ be any isomorphism
of monoidal functors, where
$\wp=(p,\mathbbm{q}\,,\varphi),\bar{\wp}=(\bar{p},\bar{\mathbbm{q}},\bar{\varphi}):\Sw\to\Sw'$
are morphisms of Schreier systems. Then, for any $a\in M$, the equality
be $p(a)=\bar{p}(a)$ must hold, since $\delta_a:p(a)\to\bar{p}(a)$ is an
isomorphism in the skeletal category $\Sigma(\Sw')$ and, moreover,
$\delta_a\in A'_{p(a)}$. Any element $f\in A_a$ defines a morphism
$f:a\to a$ in $\Sigma(\Sw)$, and the naturality of $\delta$ implies
the commutativity of the square in $\Sigma(\Sw')$
$$
\xymatrix@R=16pt{p(a)\ar[r]^{q_a(f)}\ar[d]_{\delta_a}&p(a)\ar[d]^{\delta_a}\\p(a)\ar[r]^{\bar{q}_a(f)}&p(a),}
$$
whence $\delta_a^{-1}\circ \bar{q}_a(f)\circ \delta_a=q_a(f)$. That
is, condition $(\ref{cbms1})$ in order for the family
$\delta={\big(\delta_a\in A'_{p(a)}\big)_{a\in M}}$ to be a deformation
of Schreier system morphisms from $\wp$ to $\bar{\wp}$, holds. Furthermore,
for any $a,b\in M$, the coherence condition $(\ref{cfnt})$ for
$\delta:\Sigma(\wp)\Rightarrow\Sigma(\bar{\wp})$ gives the
commutativity of
$$
\xymatrix@R=16pt{p(a)\,p(b)\ar[r]^{\varphi_{a,b}}\ar[d]_{p(a)_*(\delta_b)\circ p(b)^*(\delta_a)}&p(ab)
\ar[d]^{\delta_{ab}}\\p(a)\,p(b)\ar[r]^{\bar{\varphi}_{a,b}}&p(ab),}
$$
whence condition $(\ref{cbms2})$ follows. Therefore,
$\delta=\big(\delta_a)_{a\in M}:\wp\Rightarrow\bar{\wp}$ is actually
a deformation in $\mathbf{Z^3_{\mathrm{n\text{-}ab}}Mnd}$, and
clearly $\Sigma(\delta)=\delta$.

\subsubsection{The functor $\Sigma$ in $(\ref{lsigma})$ is essentially surjective} Suppose
$F=(F,\varphi):\Sigma(\Sw)\to\Sigma(\Sw')$ is any given monoidal
functor, where $\mathcal{S}= (M,\AAw,\Theta,\lambda)$ and
$\mathcal{S'}= (M',\AAw',\Theta',\lambda')$ are Schreier systems. By
Lemma \ref{uslem}, there is no loss of generality in assuming that
$F$ is strictly unitary, that is, $\varphi_0:1\to F(1)$ is the
identity isomorphism.

If we denote by $p:M\to M'$ the map given by the action of the
monoidal functor $F$ on objects, that is,  $p(a)=F(a)$ for any $a\in
M$, then the action of the functor $F$ on morphisms can be written,
for any $a\in M$ and $f\in A_a$, in the form
$$
F\big(a\overset{f}\to a\big)= \big(p(a)\overset{q_a(f)}\longrightarrow p(a)\big)
$$
for a map $q_a:A_a\to A'_{p(a)}$, which is indeed  a group
homomorphism since $F$ is a functor.
 Let  $\mathbbm{q}=\big(q_a:A_a\to A'_{p(a)}\big)_{a\in
M}$ denote the family of these group homomorphisms. Since the
category $\Sigma(\Sw')$ is skeletal and we have the structure
isomorphism $\varphi_{a,b}:p(a)\,p(b)\to p(ab)$ and $\varphi_0:1\to
p(1)$, it must be $p(a)p(b)=p(ab)$ and $p(1)=1$. Therefore, $p$ is a
homomorphism of monoids.

The triplet $\wp=(p,\mathbbm{q},\varphi)$ so obtained, where
$\varphi=\big(\varphi_{a,b}\in A'_{p(ab)}\big)_{a,b\in M}$,
 is actually a morphism of Schreier systems
$\wp:\Sw\to\Sw'$ and, by construction, ${\Sigma(\wp)=F}$. In effect, the
naturality of the isomorphisms $\varphi_{a,b}:p(a)\,p(b)\to p(ab)$
gives the commutativity of the squares $(\ref{tsqus})$, whence
condition $(\ref{cms1})$ holds. Moreover,  condition $(\ref{cms3})$
follows from the coherence condition $(\ref{ecf})$ which, in this
case, simply states that the diagrams $(\ref{conefc})$ are commutative.
The normalization condition $(\ref{cms4})$, $\varphi_{1,1}=1$, is a
consequence of the coherence squares $(\ref{ecf2})$, since $F$ is
assumed to be strictly unitary, that is, since $\varphi_0=1$.

\subsubsection{The monoidal equivalence $(\ref{psigma})$} We keep the notations used in  Subsection  \ref{ssamg} to define the Schreier
system $\Delta(\G)$. The mapping
$$
(a\overset{f}\to a)\ \mapsto \ (X_a\overset{f}\to X_a)
$$
is easily recognized as an equivalence of categories
$\xymatrix@C=17pt{J_\G:\Sigma(\Delta(\G))\ar[r]\ar@{}@<-2.5pt>[r]^(.7){\thicksim}&
\G}$, which, by Proposition \ref{saa}, defines a strictly unitary
monoidal equivalence when it is endowed  with the family of
isomorphisms $\varphi_{a,b}=\Gamma_{X_aX_b}:X_aX_b\to X_{ab}$,
$a,b\in M(\G)$. Note that their required naturality holds since, for
any $a,b\in M(\G)$, $g\in \mathrm{Aut}_\G(X_{a})$, and
$f\in\mathrm{Aut}_\G(X_{b})$, we have the commutative diagram
$$
\xymatrix@R=20pt{
X_aX_b\ar@{}[dr]|{(\ref{dab})}\ar@{}@<10pt>[rr]|{(B)}\ar@/^1.6pc/[rr]^{gf}
\ar[d]_{\Gamma}\ar[r]^{g1}&X_aX_b\ar@{}[dr]|{(\ref{dab})}\ar[d]^{\Gamma}\ar[r]^{1f}&X_aX_b\ar[d]^{\Gamma}\\
X_{ab}\ar@{}@<-10pt>[rr]|{=}\ar[r]_{b^*(g)}\ar@/_1.5pc/[rr]_{a_*(f)\circ b^*(g)}&X_{ab}\ar[r]_{a_*(f)}&X_{ab}}
$$
where the commutativity of the region labelled $(B)$ is a
consequence of  the fact that $\otimes:\G\times\G\to\G$ is a
functor. The needed coherence conditions (\ref{ecf}) and
(\ref{ecf2}) follow from the commutativity of diagrams
$(\ref{dlam})$ and the choice of the morphisms $\Gamma$'s made in
(\ref{sel}), respectively. \qed
\end{proof}

\subsection{The Schreier system construction biequivalence} The
above stated biequivalence between the $2$-category
 of monoidal groupoids and the $2$-category of Schreier systems
 $(\ref{tsigma})$,
$\xymatrix@C=16pt{\Sigma:\mathbf{Z^3_{\mathrm{n\text{-}ab}}Mnd}
\ar[r]\ar@{}@<-2.5pt>[r]^(.54){\approx}& \mathbf{MonGpd}}$, is
injective on objects, morphisms, and deformations. Moreover, for any
Schreier system $\Sw$, the equality $\Delta\Sigma(\Sw)=\Sw$ holds.
Hence, the assignment $\G\mapsto \Delta(\G)$, given by the Schreier
system construction $(\ref{assch})$, is the function on objects of a
biequivalence, quasi-inverse of $\Sigma$,
\begin{equation}\label{tdelta}\xymatrix@C=16pt{\Delta:\mathbf{MonGpd}
\ar[r]\ar@{}@<-2.5pt>[r]^(.5){\approx}& \mathbf{Z^3_{\mathrm{n\text{-}ab}}Mnd},}
\end{equation} which is
uniquely determined up to pseudo-natural equivalence by the equality
${\Delta\Sigma=1_{\mathbf{Z^3_{\mathrm{n\text{-}ab}}Mnd}}}$ and the
existence of a pseudo-natural equivalence
$$\xymatrix@C=11pt{J:\Sigma\Delta\ar@{=>}[r]
\ar@{}@<-1.8pt>[r]^(.39){\thicksim}&1_{\mathbf{MonGpd}},}$$ whose
component at any monoidal groupoid $\G$,
$\xymatrix@C=12pt{J_\G:\Sigma\Delta(\G)\ar[r]\ar@{}@<-2.2pt>[r]^(.65){\thicksim}&
\G}$, is the monoidal equivalence $(\ref{psigma})$.
 For completeness, we shall next show how the pseudo-functor
$\Delta$ and the pseudo-equivalence $J$
 work.

\subsubsection{$\Delta$ on monoidal functors} Suppose $F:\G\to\G'$ is any given
monoidal functor between monoidal groupoids $\G$ and $\G'$. Let
$F^{\mathrm u}:\G\to \G'$ be the  strictly unitary monoidal functor
associated to $F$ by the normalization functor described in
$(\ref{norm})$,  and let $(\Gamma_X:X\to X_a)_{a\in M(\G)}$ and
${(\Gamma'_{X'}:X'\to X'_{a'})_{a'\in M(\G')}}$ be the  cleavages
used for constructing the Schreier systems $\Delta(\G)$ and
$\Delta(\G')$, respectively. Then,
$$
\Delta(F)=(p(F),\mathbbm{q}(F)\,,\varphi(F)):\Delta(\G)\to\Delta(\G')
$$
is the morphism of Schreier systems where:
\begin{itemize}
\item[$\bullet$] $p=p(F):M(\G)\to M(\G')$ is the homomorphism of monoids defined by
$$p(a)=[FX_a]=[F^{\mathrm u}\!X_a],\hspace{0.4cm} a\in M(\G).$$
\item[$\bullet$]
$\mathbbm{q}=\mathbbm{q}(F)=\big(\mathrm{Aut}_\G(X_a)\overset{q_a}\longrightarrow
\mathrm{Aut}_{\G'}(X'_{p(a)})\big)_{a\in M(\G)}$ is the family
of group homomorphisms which carry an automorphism $f:X_a\to
X_a$, for any $a\in M(\G)$, to the unique automorphism
$q_a(f):X'_{p(a)}\to X'_{p(a)}$ in $\G'$ making the square below
commutative.
$$
\xymatrix@R=20pt{F^{\mathrm u}\!X_a\ar[r]^{F^{\mathrm u}f}\ar[d]_{\Gamma'}&F^{\mathrm u}\!X_a\ar[d]^{\Gamma'}\\
X'_{p(a)}\ar@{.>}[r]^{q_a(f)}&X'_{p(a)}
}
$$
\item[$\bullet$] $\varphi=\varphi(F)=\big(\varphi_{a,b}\in \mathrm{Aut}_{\G'}(X'_{p(ab)}\big)_{a,b\in
M(\G)}$ is the family of automorphisms in $\G'$ determined by
the commutativity of the diagrams
$$
\xymatrix{ F^{\mathrm u}\!X_a\,F^{\mathrm u}\!X_b\ar[r]^{\varphi^{\mathrm u}}\ar[d]_{\Gamma'\Gamma'}&F^{\mathrm u}(X_aX_b)\ar[r]^{F^{\mathrm u}\Gamma}&F^{\mathrm u}\!X_{ab}
\ar[d]^{\Gamma'}\\X'_{p(a)}X'_{p(b)}\ar[r]^{\Gamma'}&X'_{p(a)p(b)}\ar@{.>}[r]^{\varphi_{a,b}}&X'_{p(ab).}
}
$$
\end{itemize}
\subsubsection{$J$ on monoidal functors}
The component of the pseudo-natural equivalence $J:\Sigma
\Delta\overset{_\thicksim\ }\Rightarrow 1$ at any monoidal functor
$F:\G\to\G'$, is the isomorphism
$$
\xymatrix@R=18pt@C=40pt{\Sigma\Delta(\G)\ar[r]^{\Sigma\Delta(F)}\ar[d]_{J_\G}\ar@{}[dr]|{\overset{_\thicksim\
}\Rightarrow}&\Sigma\Delta(\G')\ar[d]^{J_{\G'}}\\ \G \ar[r]_{F}&\G'}
$$ defined by the isomorphisms of the cleavage in $\G'$, $\Gamma':F\!X_a\overset{_\thicksim\ }\to X'_{p(a)}$, $a\in M(\G)$.

\subsubsection{$\Delta$ on morphisms of monoidal functors} Let  $F,\bar{F}:\G\to\G'$ be monoidal functors as above, and suppose $\delta:F\Rightarrow \bar{F}$ is any morphism
between them. Then,
$$\xymatrix @C=18pt{\Delta(\G)
\ar@/^0.8pc/[rr]^{\Delta(F)}
 \ar@/_0.8pc/[rr]_{\Delta(\bar{F})}\ar@{}[rr]|{\Downarrow\,\Delta(\delta)} &
&\Delta(\G')}$$
is the deformation $\Delta(\delta)=\big(\Delta(\delta)_a\in
\mathrm{Aut}_{\G'}(X'_{p(a)})\big)_{a\in M(\G)}$, consisting of the
automorphisms in $\G'$ determined by the commutativity of the
diagrams below.
$$
\xymatrix@R=18pt{F^{\mathrm u}\!X_a\ar[r]^{\Gamma'}\ar[d]_{\delta^{\mathrm u}}&X'_{p(a)}\ar@{.>}[d]^{\Delta(\delta)_a}\\
\bar{F}^{\mathrm u}\!X_a\ar[r]^{\Gamma'}&X'_{p(a)}
}
$$

Since $\xymatrix@C=16pt{\Delta:\mathbf{MonGpd}
\ar[r]\ar@{}@<-2.5pt>[r]^(.5){\approx}&
\mathbf{Z^3_{\mathrm{n\text{-}ab}}Mnd}}$ is a biequivalence and, by
Lemma \ref{iseqin} every equivalence in
$\mathbf{Z^3_{\mathrm{n\text{-}ab}}Mnd}$ is actually an isomorphism,
we have the following theorem as a corollary:

\begin{theorem}\label{mtft}
$(i)$ For any Schreier system $(M,\mathbb{A},\Theta,\lambda)$, there
is a monoidal groupoid $\G$ with an isomorphism $\Delta(\G)\cong
(M,\mathbb{A},\Theta,\lambda)$.

$(ii)$ Two monoidal groupoids $\G$ and $\G'$ are equivalent if and
only if their associated Schreier systems $\Delta(\G)$ and
$\Delta(\G')$ are isomorphic.
\end{theorem}

\section{Classification of monoidal abelian groupoids.}
This section focuses on the special case  of {\em monoidal abelian
groupoids}, that is, monoidal groupoids
$\G=(\G,\otimes,\mathrm{I},\boldsymbol{a},\boldsymbol{l},\boldsymbol{r})$
whose isotropy groups $\Aut_\G(X)$, $X\in\mathrm{Ob}\G$, are all
abelian. First, we shall observe that some of the isotropy groups
of any monoidal groupoid are always abelian:
\begin{proposition}\label{proab}
$(i)$ If $\mathcal{S}= (M,\AAw,\Theta,\lambda)$ is any Schreier
system, then, for any invertible element $a\in M$, the group $A_{a}$
is abelian.

 $(ii)$ If
$\G=(\G,\otimes,\mathrm{I},\boldsymbol{a},\boldsymbol{l},\boldsymbol{r})$
is any monoidal groupoid, then, for any invertible object
$X\in \G$, the group $\Aut_\G(X)$ is abelian.
\end{proposition}
\begin{proof} $(i)$ The group $A_1$ is abelian due to
 conditions $(\ref{c3})$ and
$(\ref{c4})$: for any $f,g\in A_1$,
$$f\circ g=1_*(f)\circ 1^*(g)= 1^*(g)\circ 1_*(f)= g\circ f.$$
 For any invertible element $a\in M$, the homomorphism $a_*:A_1\to
A_a$ is actually an isomorphism, with inverse $(a^{-1})_*:A_a\to A_1$,
since, by $(\ref{c11})$, $(\ref{c5})$, and $(\ref{c4})$, we have
$$a_*(a^{-1})_*=(aa^{-1})_*=1_*=1_{A_1}, \hspace{0.4cm} (a^{-1})_*a_*=(a^{-1}a)_*=1_*=1_{A_a}.$$
Hence, $A_a$ is abelian as is $A_1$.

$(ii)$ Let  $\Delta(\G)$ be the  Schreier system associated to the
monoidal groupoid $\G$, as in (\ref{assch}). If $X\in \G$ is any
invertible object, then $a=[X]\in M(\G)$ is an invertible element of
the associated monoid $(\ref{m})$, whence, due to part $(i)$, the
group $\mathrm{Aut}_\G(X_a)$ is abelian. Since the isomorphism
$\Gamma:X\to X_a$ induces a group isomorphism
$\mathrm{Aut}_\G(X)\cong \mathrm{Aut}_\G(X_a)$, the result
follows.\qed
\end{proof}

Therefore, for example, every categorical group is a monoidal abelian
groupoid. The classification of categorical groups was given by Sinh
in \cite{Si}, by means of Eilenberg-Mac Lane group cohomology groups
$\HH^3(G,(A,\theta))$, and our aim here is to give a similar solution
to the more general problem of classifying the monoidal abelian groupoids, now by means of monoid
cohomology groups $\HH^3(M,(\AAw,\Theta))$. To this end, we shall
briefly review below some basic aspects concerning the cohomology
theory of monoids that we are going to use, which is a
generalization of Eilenberg-Mac Lane's cohomology of groups due to
Leech \cite{leech}.
\subsection{Leech cohomology of monoids}
If $\mathcal{C}$ is a small category, then the category of (left)
$\mathcal{C}$-modules has objects the functors $D:\mathcal{C}\to
\mathbf{Ab}$ from $\mathcal{C}$ into abelian groups, with morphisms
the natural transformations. This is an abelian category with enough
injectives and projectives, and the abelian groups
$\HH^n(\mathcal{C},D)=\mathrm{Ext}_{\mathcal{C}}^n(\mathbb{Z},D)$,
where $\mathbb{Z}:\mathcal{C}\to \mathbf{Ab}$ is the constant
functor with value $\mathbb{Z}$, are the cohomology groups of the
category $\mathcal{C}$ with coefficients in the $\mathcal{C}$-module
$D$, studied by Roos \cite{Ro} and Watts \cite{W}, among other
authors. Cohomology theory of small categories is in itself a basis for
other cohomology theories, in particular for the Leech cohomology theory
of monoids, which is defined as follows:

A monoid $M$ gives rise to a category $\mathbb{D}\!M$ with object
set $M$ and arrow set $M\times M\times M$, with $(a,b,c):b\to abc$.
Composition is given by $$(d,abc,e)(a,b,c)=(da,b,ce),$$ and the
identity morphism of any object $a$ is $1_a=(1,a,1)$, where $1$ is
the unit element of $M$. This construction $M\mapsto \mathbb{D}M$
defines a functor $\mathbb{D}:\mathbf{Mnd}\to \mathbf{Cat}$, which
maps a monoid homomorphism  $p:M\to M'$ to the functor
$\mathbb{D}p:\mathbb{D}M\to \mathbb{D}M'$ given by $
\mathbb{D}p(a,b,c)=(p(a),p(b),p(c))$.

If we say that a $\DD M$-module, $\mathbb{D}M\to \mathbf{Ab}$,
carries the element $a\in M$ to the abelian group $A_a$ and carries
the morphism $(a,b,c)$ to the group homomorphism $a_*c^*:A_b\to
A_{abc}$, then we see that such a $\DD M$-module, hereafter denoted
by
$$
(\AAw,\Theta):\DD M\to \mathbf{Ab},
$$ is a
system of data consisting of two families of abelian groups and
homomorphisms, respectively,
$$\AAw=(A_{a})_{a\in M},\hspace{0.4cm} \Theta=(
A_b\overset{a_*}\longrightarrow A_{ab} \overset{b^*}\longleftarrow
A_a)_{a,b\,\in M}$$ such that, for any $a,b,c\in M$,
$$
(ab)_*=a_*b_* :A_c\to A_{abc},\
 c^*a_*=a_*c^*:A_b\to A_{abc},\
c^*b^*=(bc)^*:A_a\to A_{abc},  \
$$
and, for any $a\in M$, $$ 1_*=1_{A_a}=1^*:A_a\to A_a.$$

{\em Leech cohomology groups} $\HH^n(M,(\AAw,\Theta))$ \cite{leech},
of a monoid $M$ with coefficients in a $\DD M$-module
$(\AAw,\Theta)$, are defined to be those of its associated category
$\mathbb{D}M$, that is,
$$
\HH^n(M,(\AAw,\Theta))=\HH^n(\DD M,(\AAw,\Theta)).
$$

For  computing these cohomology groups, there is a cochain complex
\begin{equation}\label{compcal}C^\bullet(M,(\AAw,\Theta)),\end{equation} which is defined in degree $n>0$ by
$$
C^n(M,(\AAw,\Theta))=\left\{\lambda\in\prod\limits_{(a_1,\ldots,a_n)\in M^n}
\hspace{-0.7cm}A_{a_1\cdots a_n}\mid \lambda_{a_1,\ldots,a_{n}}=1 \text{ whenever some } a_i=1\right\}
$$
and $C^0(M,(\AAw,\Theta))=A_1$. The coboundary operator
$$\partial^n: C^n(M,(\AAw,\Theta))\to C^{n+1}(M,(\AAw,\Theta))$$ is
given, for $n=0$,  by $(\partial^0\lambda)_a=a_*(\lambda) \circ
a^*(\lambda)^{-1}$, while, for $n>0$,

$$
(\partial^n\lambda)_{a_1,\ldots,a_{n+1}}= (a_1)_*(\lambda_{a_2,\ldots,a_{n+1}})
\circ \prod\limits_{i=1}^n \lambda_{a_1,\ldots,a_ia_{i+1},\ldots,a_{n+1}}^{(-1)^{i}}
\circ a_{n+1}^*(\lambda_{a_1,\ldots,a_{n}}^{(-1)^{n+1}}).
$$

\noindent By \cite[Chapter II, 2.3, 2.9]{leech}, we have
$$ \HH^n(M,(\AAw,\Theta))=\HH^n\big(C^\bullet(M,(\AAw,\Theta))\big). $$

It is useful for our purposes to describe the natural
properties of the Leech cohomology on the category obtained by the
 Grothendieck construction on the  functor
 that associates to any monoid $M$
the category of $\DD M$-modules and, to any homomorphism $p:M\to
M'$, the functor $p^*$ that carries any $\DD M'$-module, say
$(\AAw',\Theta')$, to the $\DD M$-module
$$p^*(\AAw',\Theta')=(p^*\AAw',p^*\Theta'),$$where
$$p^*\AAw'=(A'_{p(a)})_{a\in M},\hspace{0.4cm}p^*\Theta'=\xymatrix{(
A'_{p(b)}\ar[r]^{p(a)_*}&A'_{\!p(ab)}
&\ar[l]_{p(b)^*}A'_{p(a)})_{a,b\,\in M}}.
$$

By applying the Grothendieck construction on the functor $M\mapsto
\DD M$-modules, we get a category, denoted by
$$
\mathcal{M}od_\DD,
$$
which may be heuristically viewed as the category obtained by tying
the categories of $\DD M$-modules together in a  natural fashion. It
 has objects pairs $(M,(\AAw,\Theta))$,  where $M$ is a monoid and $(\AAw,\Theta)$ is a
$\DD M$-module. Morphisms are pairs
$$(p,\mathbbm{q}):(M,(\AAw,\Theta))\to (M',(\AAw',\Theta'))$$ consisting
of a  monoid homomorphism $p:M\to M'$ together with a morphism of
$\DD M$-modules $\mathbbm{q}:(\AAw,\Theta)\to p^*(\AAw',\Theta')$,
that is, a family of group homomorphisms
$$\mathbbm{q}=\xymatrix@C=15pt{\big(A_a\ar[r]^-{q_a}&
A_{p(a)}^\prime\big)_{\!a\in M}},$$ satisfying, for any $a,b\in
M$,
\begin{equation}\label{yano}
q_{ab}\,a_*=p(a)_*\,q_b:A_b\to A'_{p(ab)},
\end{equation}
\begin{equation}\label{yano2}
q_{ab}\,b^*=p(b)^*\,q_a:A_a\to A'_{p(ab)}.
\end{equation}
Composition is defined by
$(p',\mathbbm{q}')(p,\mathbbm{q})=(p'p,\mathbbm{q}'\mathbbm{q})$,
where $\mathbbm{q}'\mathbbm{q}=(q'_{p(a)}q_a)_{a\in M}$.

Any morphism $(p,\mathbbm{q}):(M,(\AAw,\Theta))\to
(M',(\AAw',\Theta'))$ as above yields homomorphisms
$$
\HH^n(M,(\AAw,\Theta))\overset{\mathbbm{q}_*}\longrightarrow \HH^n(M, p^*(\AAw',\Theta'))
\overset{p^*}\longleftarrow
\HH^n(M',(\AAw',\Theta'))
$$
induced by the morphisms of cochain complexes
$$
C^\bullet(M,(\AAw,\Theta))\overset{\mathbbm{q}_*}\longrightarrow C^\bullet(M,p^*(\AAw',\Theta'))
\overset{p^*}\longleftarrow
C^\bullet(M',(\AAw',\Theta')),
$$
which are given on cochains by
$$(\mathbbm{q}_*\lambda)_{a_1,\ldots,a_n}= q_{a_1\cdots
a_n}(\lambda_{a_1,\ldots,a_n}), \hspace{0.3cm}
(p^*\lambda')_{a_1,\ldots,a_n}=\lambda'_{p(a_1),\ldots,p(a_n)}.$$

\subsubsection{Leech cohomology versus Gabriel-Zisman cohomology}
 Cohomology theory of small categories is also the foundation for the (co)homology theory of
simplicial sets with twisted coefficients, as  defined by Gabriel
and Zisman in \cite[Appendix II]{G-Z} (actually, the results by
Gabriel-Zisman are stated for homology of simplicial sets, but they
can be easily dualised to cohomology, see Illusie \cite[Chapter VI,
\S 3]{Ill}). Briefly, recall that the simplicial category, denoted
by $\mathbf{\Delta}$, has objects the ordered sets
$[n]=\{0,\ldots,n\}$, $n\geq 0$, and its arrows are weakly monotone
maps $\alpha:[m]\rightarrow[n]$. If
$X:\mathbf{\Delta}^{\!\mathrm{op}}\to \mathbf{Set}$ is a simplicial
set, then its {\em category of simplices}, denoted by
$\mathbf{\Delta}/X$, is the category obtained by the Grotendieck
construction on $X$. That is, the category whose objects are pairs
$(x,m)$ with $x\in X_m$; an arrow $\alpha:(x,m)\to (y,n)$ is a
 map $\alpha:[m]\to [n]$ in $\mathbf{\Delta}$ such that $\alpha^*y=x$,
where  we write $\alpha^*:X_n\to X_m$ for the map induced by
$\alpha$. A {\em coefficient system on $X$} is a
$\mathbf{\Delta}/X$-module, that is, a functor
$L:\mathbf{\Delta}/X\to \mathbf{Ab}$, and the {\em cohomology groups
of $X$ with coefficients in $L$} are, by definition,
$$\xymatrix{ \HH^n(X,L)=\HH^n(\mathbf{\Delta}/X,L),\hspace{3pt}n\geq 0.}$$

Suppose now $M$ is a monoid. The {\em classifying space} of $M$ is
the simplicial set $BM$, whose set of $n$-simplices is
$$
(BM)_n=\Big\{\mathbf{a}=(a_{i,j}\in M)_{0\leq i\leq j\leq n}
\mid a_{i,j}a_{j,k}=a_{i,k},\,a_{i,i}=1\Big\},
$$
which is usually identified with $M^n$ by the bijection
$\mathbf{a}\mapsto (a_{0,1},\dots,a_{n-1,n})$. The induced map
$\alpha^*:(BM)_n\to (BM)_m$, for a map $\alpha:[m]\to [n]$ in
$\mathbf{\Delta}$, is given by $\alpha^*(\mathbf{b})=\mathbf{a}$,
where $a_{i,j}=b_{\alpha(i),\alpha(j)}$.

There is a canonical functor $\mathbf{\Delta}/BM\to \DD M$, given by
$$
\xymatrix{(\mathbf{a},m)\ar[r]^{\alpha}&(\mathbf{b},n)}\ \mapsto
\xymatrix@C=85pt{a_{0,m}\ar[r]^{(b_{0,\alpha(0)},a_{0,m},b_{\alpha(m),n})}&b_{0,n}.}
$$
Therefore, by composition with it, each $\DD M$-module, say
$(\AAw,\Theta):\DD M\to \mathbf{Ab}$, gives a coefficient system on
$BM$, also denoted by $(\AAw,\Theta)$, and Gabriel-Zisman cohomology
groups of $BM$ with coefficients in the $\DD M$-module
$(\AAw,\Theta)$,
$$\HH^n(BM,(\AAw,\Theta)),$$
are defined. But note that, by \cite[Appendix II, Proposition
4.2]{G-Z}, these cohomology groups of $BM$ can be computed by the
same cochain complex $(\ref{compcal})$,
$C^\bullet(M,(\AAw,\Theta))$, used by Leech for computing the
cohomology groups of the monoid $M$ (see also \cite[Chapter VI,
(3.4.3)]{Ill}). Therefore, there is a natural identification
$$\HH^n(M,(\AAw,\Theta))=\HH^n(BM,(\AAw,\Theta)).$$

\subsection{The classification theorems} The
biequivalence in Theorem \ref{mt1} restricts to a biequivalence
between the full 2-subcategory of the 2-category of monoidal
groupoids
 given by the
monoidal abelian groupoids, denoted by
$$\mathbf{MonAbGpd},$$ and the full 2-subcategory of
the 2-category of Schreier systems given by those Schreier systems $
(M,\AAw,\Theta,\lambda)$ in which every group $A_a$, $a\in M$, is
abelian. Hereafter, this latter 2-category will be called  the {\em
$2$-category of Leech $3$-cocycles of monoids}, and be denoted by
$$\mathbf{Z^3}\mathbf{Mnd},$$ since its cells have the following
cohomological interpretation:

{\em 0-cells}. According to Definition \ref{defSch}, a Schreier
system $\mathcal{S}$ in $\mathbf{Z^3}\mathbf{Mnd}$ precisely is   a
triplet $\mathcal{S}= (M,(\AAw,\Theta),\lambda)$ consisting of a
monoid $M$, a $\DD M$-module $(\AAw,\Theta)$, and a 3-cocycle
$\lambda\in \Z^3(M,(\AAw,\Theta))$.

{\em 1-cells}. If $\mathcal{S}= (M,(\AAw,\Theta),\lambda)$ and
$\mathcal{S'}= (M',(\AAw',\Theta'),\lambda')$ are in
$\mathbf{Z^3}\mathbf{Mnd}$, then a morphism of Schreier systems (see
Subsection \ref{morscsys}),
$\wp=(p,\mathbbm{q}\,,\varphi):\Sw\to\Sw'$, is the same thing as a
morphism $(p,\mathbbm{q}):(M,(\AAw,\Theta))\to (M',(\AAw',\Theta'))$
in $\mathcal{M}od_\DD$,  together with a 2-cochain $\varphi\in
C^2(M,p^*(\AAw',\Theta'))$ such that
$\mathbbm{q}_*\lambda=p^*\lambda'\circ
\partial^2 \varphi$.

{\em 2-cells.} If $\wp=(p,\mathbbm{q},\varphi):\Sw\to\Sw'$ and
$\bar{\wp}=(\bar{p},\bar{\mathbbm{q}},\bar{\varphi}):\Sw\to\Sw'$ are
morphisms in $\mathbf{Z^3}\mathbf{Mnd}$, then (see Subsection
\ref{defmorSchr}) there is no deformation between them unless
$p=\bar{p}$ and $\mathbbm{q}=\bar{\mathbbm{q}}$. In such a case, a
deformation $\delta:\wp\Rightarrow\bar{\wp}$ consists of a 1-cochain
$\delta\in C^1(M,p^*(\AAw',\Theta'))$, such that
$\varphi=\bar{\varphi}\circ
\partial^1\delta$.

 Therefore, our first result here comes as a direct
consequence of Theorem \ref{mt1}:
\begin{theorem}\label{thebires} The
quasi-inverse biequivalences $(\ref{tsigma})$ and $(\ref{tdelta})$
restrict to corresponding quasi-inverse biequivalences
\begin{equation}\label{biequires}
\xymatrix{\mathbf{MonAbGpd}\ar@{}[r]|-{\thickapprox}
\ar@<4pt>[r]^-{\Delta}&\ar@<3pt>[l]^-{\Sigma} \mathbf{Z^3}\mathbf{Mnd}.
}
\end{equation}
\end{theorem}

Closely related to the category $\mathbf{Z^3}\mathbf{Mnd}$ is the
{\em category of Leech $3$-cohomology classes of monoids}, denoted
by
\begin{equation}\label{h3mon}\mathbf{H^3}\mathbf{Mnd},\end{equation}
which plays a fundamental role in stating our classification theorem
below. Its objects are triplets $(M,(\AAw,\Theta), c)$, where $M$ is
a monoid, $(\AAw,\Theta)$ is a $\DD M$-module, and $c\in
\HH^3(M,(\AAw,\Theta))$ is a 3-cohomology class of $M$ with coefficients
in $(\AAw,\Theta)$. An arrow $$(p,\mathbbm{q}):(M,(\AAw,\Theta),c)\to
(M',(\AAw',\Theta'),c')$$ is a morphism
$(p,\mathbbm{q}):(M,(\AAw,\Theta))\to (M',(\AAw',\Theta'))$ in
$\mathcal{M}od_\DD$, such that
$$p^*(c')=\mathbbm{q}_*(c)\in \HH^3(M,p^*(\AAw',\Theta')).$$

Observe that a morphism $(p,\mathbbm{q})$ is an isomorphism in
$\mathbf{H^3}\mathbf{Mnd}$ if and only if $p:M\to M'$  is an
isomorphism of monoids and $\mathbbm{q}:(\AAw,\Theta)\to
p^*(\AAw',\Theta') $ is an isomorphism of $\DD M$-modules.

We have the {\em cohomology class functor}
\begin{eqnarray*}\label{cohoclass}&\mathrm{cl}: \mathbf{Z^3}\mathbf{Mnd}\to
\mathbf{H^3}\mathbf{Mnd},\\[4pt] \nonumber
&(M,(\AAw,\Theta),\lambda)\mapsto (M,(\AAw,\Theta),[\lambda])\\ \nonumber
&(p,\mathbbm{q}\,,\varphi)\mapsto (p,\mathbbm{q})\end{eqnarray*}
 where
$[\lambda]\in \HH^3(M,(\AAw,\Theta))$ denotes the cohomology class of
$\lambda \in \Z^3(M,(\AAw,\Theta))$.
 This functor clearly carries two isomorphic morphisms of $\mathbf{Z^3}\mathbf{Mnd}$ to
 the same morphism in $\mathbf{H^3}\mathbf{Mnd}$, whence composition with the pseudo-functor
 $\Delta$ above gives a functor
\begin{equation}\label{eqclas}
 \mathrm{Cl}=\mathrm{cl}\,\Delta:\mathbf{MonAbGpd}\to \mathbf{H^3}\mathbf{Mnd},
\end{equation}
that we call {\em the classifying functor} because of the theorem
below.
\begin{theorem}[Classification of monoidal abelian groupoids]\label{mt3}
$(i)$ For any monoid $M$, any $\DD M$-module $(\AAw,\Theta)$, and any
cohomology class $c\in \HH^3(M,(\AAw,\Theta))$, there is a monoidal
abelian groupoid $\G$ with an isomorphism $\mathrm{Cl}(\G)\cong
(M,(\AAw,\Theta),c)$.

$(ii)$ A monoidal functor between monoidal abelian groupoids
$F:\G\to\G'$ is an equivalence if and only if $\mathrm{Cl}(F):
\mathrm{Cl}(\G)\to \mathrm{Cl}(\G')$ is an isomorphism.

$(iii)$ For any isomorphism
$(p,\mathbbm{q}):\mathrm{Cl}(\G)\cong\mathrm{Cl}(\G')$,
 there is a monoidal equivalence
 $F:\G\overset{_\sim\ }\to\G'$
 such that $\mathrm{Cl}(F)=(p,\mathbbm{q})$.

 $(iv)$ If $\G$ and $\G'$ are monoidal abelian groupoids with $\mathrm{Cl}(\G)=(M,(\AAw,\Theta),c)$ and
$\mathrm{Cl}(\G')=(M',(\AAw',\Theta'),c')$, then,  for any morphism
 $(p,\mathbbm{q}):\mathrm{Cl}(\G)\to \mathrm{Cl}(\G')$
 in  $\mathbf{H^3}\mathbf{Mnd}$, there is a
(non-natural) bijection
$$
\big\{[F]:\G\to\G' \mid \mathrm{Cl}(F)=(p,\mathbbm{q})\big\}\cong \HH^2(M,p^*(\AAw',\Theta'))
$$
between the set of isomorphism classes of those monoidal functors
$F:\G\to \G'$ that are carried by the classifying functor to
$(p,\mathbbm{q})$ and the elements of the second cohomology group of
$M$ with coefficients in the $\DD M$-module $p^*(\AAw',\Theta')$.
\end{theorem}
\begin{proof}
 $(i)$
Given any object $(M,(\AAw,\Theta),c)\in \mathbf{H^3}\mathbf{Mnd}$,
let us choose any 3-cocycle $\lambda \in \Z^3(M,(\AAw,\Theta))$ such
that $[\lambda]=c$. Then, letting $\G=\Sigma(M,\AAw,\Theta,\lambda)$,
we have
$$
\mathrm{Cl}(\G)=\mathrm{cl}(\Delta\Sigma(M,\AAw,\Theta,\lambda))=
\mathrm{cl}(M,(\AAw,\Theta),\lambda)=(M,(\AAw,\Theta),c).
$$

$(ii)$ Since the pseudo-functor $\Delta: \mathbf{MonAbGpd}\to
\mathbf{Z^3}\mathbf{Mnd}$ is a biequivalence, it suffices to prove
that a morphism in $\mathbf{Z^3}\mathbf{Mnd}$, say
$$(p,\mathbbm{q},\varphi):(M,(\AAw,\Theta),\lambda)\to
(M',(\AAw',\Theta'),\lambda'),$$ is an equivalence if and only if the
induced  $$(p,\mathbbm{q}):(M,(\AAw,\Theta),[\lambda])\to
(M',(\AAw',\Theta'),[\lambda'])$$ is an isomorphism in
$\mathbf{H^3}\mathbf{Mnd}$, that is, if and only if $p:M\to M'$ is
an isomorphism of monoids and $\mathbbm{q}:(\AAw,\Theta)\to
p^*(\AAw',\Theta')$ is an isomorphism of $\DD M$-modules. Hence, the
result follows from Lemma \ref{iseqin}.

$(iv)$ Suppose $\Delta(\G)= (M,\AAw,\Theta,\lambda)$ and
$\Delta(\G')= (M',\AAw',\Theta',\lambda')$, and let
$(p,\mathbbm{q}): \mathrm{Cl}(\G)\to \mathrm{Cl}(\G')$ be any given
morphism in $\mathbf{H^3}\mathbf{Mnd}$. The equivalence between the
hom-groupoids
$$\mathbf{MonAbGpd}(\G,\G')\overset{\Delta}\simeq
\mathbf{Z^3}\mathbf{Mnd}(\Delta(\G),\Delta(\G')),$$ induces a bijection, $[F]\mapsto [\Delta(F)]$,
$$\big\{[F]:\G\to\G' \mid \mathrm{Cl}(F)=(p,\mathbbm{q})\big\}\cong
\big\{[p,\mathbbm{q},\varphi]:(M,(\AAw,\Theta),\lambda)\to
(M',(\AAw',\Theta'),\lambda') \big\}
$$
between the set of iso-classes $[F]$ of those monoidal functors
$F:\G\to\G'$ with $\mathrm{Cl}(F)=(p,\mathbbm{q})$, and the  set of
iso-classes $[p,\mathbbm{q},\varphi]$ of morphisms of the form
$$(p,\mathbbm{q},\varphi):(M,(\AAw,\Theta),\lambda)\to
(M',(\AAw',\Theta'),\lambda')$$ in the 2-category of Leech
3-cocycles. Since $p^*[\lambda']=\mathbbm{q}_*[\lambda]$, both
3-cocycles $p^*\lambda'$ and $\mathbbm{q}_*\lambda$  represent the
same class in the cohomology group $\HH^3(M,p^*(\AAw',\Theta'))$.
Therefore,  a 2-cochain $\varphi\in
C^2(M,p^*(\AAw',\Theta'))$ such that $q_*\lambda=p^*\lambda'\circ
\partial^2\varphi$  must exist. Hence,
$(p,\mathbbm{q},\varphi):(M,(\AAw,\Theta),\lambda)\to
(M',(\AAw',\Theta'),\lambda')$ is  a morphism in
$\mathbf{Z^3}\mathbf{Mnd}$. Furthermore, observe that any other
morphism  of the form $(p,\mathbbm{q},\psi):(M,(\AAw,\Theta),\lambda)\to
(M',(\AAw',\Theta'),\lambda')$ in $\mathbf{Z^3}\mathbf{Mnd}$ realizing the same morphism
$(p,\mathbbm{q})$ of $\mathbf{H^3}\mathbf{Mnd}$ is necessarily
written in the form $(p,\mathbbm{q},\varphi\circ \phi)$ for
some $\phi\in \Z^2(M,p^*(\AAw',\Theta'))$ and, moreover,  the
morphisms $(p,\mathbbm{q},\varphi)$ and
$(p,\mathbbm{q},\varphi\circ \phi)$ are isomorphic if and only
if $\phi=\partial^1\delta$ for some $\delta\in
C^1(M,p^*(\AAw',\Theta'))$. That is, there is a bijection
$$
\HH^2(M,p^*(\AAw',\Theta'))\cong \big\{[p,\mathbbm{q},\psi]:(M,(\AAw,\Theta),\lambda)\to
(M',(\AAw',\Theta'),\lambda') \big\}
$$
given by $[\phi]\mapsto [p,\mathbbm{q},\varphi\circ \phi]$.

$(iii)$ Let $(p,\mathbbm{q}): \mathrm{Cl}(\G)\cong \mathrm{Cl}(\G')$
be any given isomorphism in $\mathbf{H^3}\mathbf{Mnd}$. By the
already proved part $(iv)$, there exists a monoidal functor
$F:\G\to\G'$ such that $\mathrm{Cl}(F)=(p,\mathbbm{q})$, which, by
part $(ii)$ is an equivalence.\qed
\end{proof}

The functor $ \mathbf{MonAbGpd}\to\mathcal{M}od_\DD$, $\G\mapsto
(M(\G), (\AAw(\G),\Theta(\G))$, obtained by composing the classifying
functor $(\ref{eqclas})$
 with the
forgetful functor ${\mathbf{H^3}\mathbf{Mnd}\to\mathcal{M}od_\DD}$,
$(M,(\AAw,\Theta),c)\mapsto (M,(\AAw,\Theta))$, turns the 2-category
of monoidal abelian groupoids into a fibred 2-category over the
category $\mathcal{M}od_\DD$. It follows from the above results
that, for any fixed monoid $M$ and  $\DD M$-module $(\AAw,\Theta)$,
the mappings $$[\lambda]\mapsto
[\Sigma(M,\AAw,\Theta,\lambda)],\hspace{0.3cm} \G\mapsto
[\lambda(\G)],$$ describe mutually inverse bijections between the
set $\HH^3(M,(\AAw,\Theta))$ and the set of equivalence classes of
monoidal groupoids in the fibre 2-category
 over $(M,(\AAw,\Theta))$. However, this latter set  is conceptually a little too rigid
since the strict requirements $M(\G)=M$ and
$(\AAw(\G),\Theta(\G))=(\AAw,\Theta)$, for a monoidal abelian groupoid
$\G$, are not very natural. We shall show how to relax them below.

\begin{definition} \label{addef} For any given monoid $M$ and any $\DD M$-module $(\AAw,\Theta)$, we say that a monoidal abelian groupoid $\G$ is {\em of type} $(M,(\AAw,\Theta))$ if there are given
\begin{itemize}
\item[$\bullet$] a monoid isomorphism $i:M\cong M(\G)$,
\item[$\bullet$] a family of group isomorphisms
$\mathbbm{j}=\big(j_X:A_{a}\cong\mathrm{Aut}_{\G}(X)\big)_{\!a\in
M, \,X\in i(a)}$,
\end{itemize}
such that,

$\bullet$ if $X,Y\in i(a)$ then,  for any morphism $h:X\to Y$ in
$\G$ and any $g\in A_{a}$,
$$j_Y(g)=h\circ j_X(g)\circ h^{-1}.
$$

$\bullet$ if $X\in i(a)$ and $Y\in i(b)$, then, for any  $f\in A_b$
and  $g\in A_a$,
$$
j_{XY}(a_*(f))=1_X\,j_Y(f),\hspace{0.3cm} j_{XY}(b^*(g))=j_X(g)\,1_Y.
$$

If $(p,\mathbbm{q}):(M,(\AAw,\Theta))\to (M',(\AAw',\Theta'))$ is any
morphism in the category $\mathcal{M}od_\DD$,   and  $\G$ and  $\G'$
are monoidal abelian groupoids of the respective types
$(M,(\AAw,\Theta))$ and  $(M',(\AAw',\Theta'))$, then a monoidal
functor $F:\G\to\G'$ is said to be {\em of type} $(p,\mathbbm{q})$
whenever

 $\bullet$  if $X\in i(a)$, then $FX\in i'(p(a))$,
and, for any $g\in A_a$, $j'_{FX}q_{a}(g)=F(j_X(g))$.

 \vspace{0.2cm}Two monoidal abelian groupoids of the same
  type $(M,(\AAw,\Theta))$, say $(\G,i,\mathbbm{j})$ and $(\G',i',\mathbbm{j}')$,  are defined to be
  {\em equivalent} if there exists a monoidal equivalence $F:\G\to\G'$
of type $(1,\mathbbm{1})$, that is, whenever

$\bullet$  if $X\in i(a)$, then $FX\in i'(a)$, and, for any $g\in
A_a$, $ j'_{FX}(g)=F(j_X(g)) $.
\end{definition}

If we denote by $$\mathbf{MonAbGpd}(M,(\AAw,\Theta))$$ the set of
equivalence classes $[\G,i,\mathbbm{j}]$ of those monoidal abelian
groupoids $(\G,i,\mathbbm{j})$ of type $(M,(\AAw,\Theta))$, then we
are ready to summarize our results on the classification of monoidal
abelian groupoids and their homomorphisms in slightly more classic
terms:

\begin{theorem}\label{concl}
$(i)$ For any monoidal abelian groupoid $\G$, there exists a monoid
$M$ and a $\DD M$-module $(\AAw,\Theta)$ such that $\G$ is of type
$(M,(\AAw,\Theta))$.

$(ii)$ For any monoid $M$ and any $\DD M$-module $(\AAw,\Theta)$,
there is a natural bijection
\begin{equation}\label{eqbij}\begin{array}{c}\mathbf{MonAbGpd}(M,(\AAw,\Theta))\cong
\HH^3(M,(\AAw,\Theta))
\end{array}\end{equation} given by   $$[\G,i,\mathbbm{j}]\ \mapsto\
c(\G)=\mathbbm{j}_*^{-1}i^*([\lambda(\G)]),$$ where $\lambda(\G)$ is
the $3$-cocycle obtained as in $(\ref{abc})$, and
$$
\HH^3(M(\G),(\AAw(\G),\Theta(\G)))\overset{i^*}\longrightarrow \HH^n(M, i^*(\AAw(\G),\Theta(\G)))
\overset{\mathbbm{j}_*^{-1}}\longrightarrow
\HH^3(M,(\AAw,\Theta))
$$
the induced isomorphisms on cohomology groups by the isomorphism
$$(i,\mathbbm{j}):(M,(\AAw,\Theta))\cong (M(\G),
(\AAw(\G),\Theta(\G))$$
in the category $\mathcal{M}od_\DD$.  In the other direction, the
bijection is induced by the mapping that carries a $3$-cocycle
$\lambda\in \Z^3(M,(\AAw,\Theta))$ to the monoidal abelian groupoid
$\Sigma(M,\AAw,\Theta,\lambda)$, given by the construction
$(\ref{gc})$.

$(iii)$ If $\G$ is of type $(M,(\AAw,\Theta))$ and $\G'$ is of type
$(M',(\AAw',\Theta'))$, then for every monoidal functor, $F:\G\to
\G'$, there exists a morphism in the category $\mathcal{M}od_\DD$,
${(p,\mathbbm{q}):(M,(\AAw,\Theta))\to (M',(\AAw',\Theta'))}$, such
that $F$
 is of type $(p,\mathbbm{q})$.

$(iv)$ If $\G$ is of type $(M,(\AAw,\Theta))$ and $\G'$ is of type
$(M',(\AAw',\Theta'))$, then, for any morphism
$(p,\mathbbm{q}):(M,(\AAw,\Theta))\to (M',(\AAw',\Theta'))$  in the
category $\mathcal{M}od_\DD$, there is a monoidal functor
$F:\G\to \G'$ of type $(p,\mathbbm{q})$ if and only if
$$p^*(c(\G'))=\mathbbm{q}_*(c(\G))\in \HH^3(M,p^*(\AAw',\Theta')).$$

In such a case, isomorphism classes of  monoidal functors
$F:\G\to\G'$ of type $(p,\mathbbm{q})$ are in bijection with the
elements of the group
$$
\HH^2(M,p^*(\AAw',\Theta')).
$$
\end{theorem}
\begin{proof} All the statements here are a direct consequence of those in Theorem \ref{mt3} after two quite obvious observations, namely:
  (1) A monoidal abelian groupoid $\G$ is of type $(M,(\AAw,\Theta))$
  if and only if there is given an isomorphism
$$(i,\mathbbm{j}):(M,(\AAw,\Theta))\cong (M(\G), (\AAw(\G),\Theta(\G))$$
in the category $\mathcal{M}od_\DD$. (2) If
$(p,\mathbbm{q}):(M,(\AAw,\Theta))\to (M',(\AAw',\Theta'))$ is a
morphism in the category $\mathcal{M}od_\DD$, and  $\G$ and $\G'$ are
any monoidal groupoids
   of respective types  $(M,(\AAw,\Theta))$ and $(M',(\AAw',\Theta'))$, then a monoidal functor $F:\G\to\G'$ is of type $(p,\mathbbm{q})$ if and only if the square below in the category $\mathcal{M}od_\DD$ commutes.
$$
\xymatrix{(M,(\AAw,\Theta))\ar[r]^(.4){(i,\,\mathbbm{j})}\ar[d]_{(p,\mathbbm{q})}& (M(\G), (\AAw(\G),\Theta(\G))\ar[d]^{(p(F),\mathbbm{q}(F))}\\
(M',(\AAw',\Theta'))\ar[r]^(.4){(i'\!,\,\mathbbm{j}')}&(M(\G'), (\AAw(\G'),\Theta(\G')).
}
$$ \qed
\end{proof}

\begin{remark} The category of monoids is tripleable over  the category of sets. In
\cite[Theorem 8]{wells}, Wells identified the category $
\mathbf{Ab}(\mathbf{Mnd}\!\downarrow_{\! M})$ of  abelian group
objects in the comma category of monoids over a monoid $M$ with  the
category of $\DD M$-modules (see Example \ref{exagp}), and he proved
that with a dimension shift both the Barr-Beck cotriple cohomology
theory \cite{Barr-Beck,Beck} and the Leech cohomology theory of
monoids are the same. Hence, for any monoid $M$ and any $\DD
M$-module $(\AAw,\Theta)$, the  Duskin \cite{Duskin} and Glenn
\cite{Glenn} general interpretation theorem for cotriple cohomology
classes shows that
 equivalence classes of 2-{\em torsors} over $M$ under
$(\AAw,\Theta)$ are in bijection with elements of the cohomology
group $\HH^{3}(M,(\AAw,\Theta))$.

A very similar result  follows from the general result by Pirashvili
\cite{Pir1,Pir2} and Baues-Dreckmann \cite{B-D} about the
classification of  track categories. From this result, the elements
of $\HH^{3}(M,(\AAw,\Theta))$ are in bijection with equivalence
classes of {\em linear track extensions} of (the category) $M$ by
the $\DD M$-module ({\em natural system on $M$} in their
terminology) $(\AAw,\Theta)$.

Indeed, the three terms `2-torsor over $M$ under $(\AAw,\Theta)$',
`linear track extension of $M$ by $(\AAw,\Theta)$', and `{\em strict}
monoidal abelian groupoid of type $(M,(\AAw,\Theta))$' are plainly
recognized to be synonymous: simply take into account that an
internal groupoid in the category of monoids is the same thing as a
strict monoidal groupoid, together with Lemmas 2.2 and 2.3 in
\cite{C-A} (or \cite[Theorem 3.3]{C-B-G}).

However, we must stress that while it is relatively harmless to
consider monoidal abelian groupoids as `strict', since by the Mac
Lane Coherence Theorem for monoidal categories \cite{Mac,Jo-St}
every monoidal abelian groupoid is equivalent to a strict one, we
consider it is not so harmless when dealing with their homomorphisms
since {\em not every monoidal functor is isomorphic to a strict
one}. Indeed, it is possible to find two strict monoidal abelian
groupoids, say $\G$ and $\G'$, that are related by a monoidal
equivalence between them but there is no strict equivalence either
from $\G$ to $\G'$ nor from $\G'$ to $\G$. For this reason, if to
establish the bijection $(\ref{eqbij})$, we want to use only strict
monoidal abelian groupoids and strict equivalences between them, as
we need to do for applying Duskin or Pirashvili classification
results, then we must define two strict monoidal abelian groupoids
$\G$ and $\G'$ as {\em equivalent} if there is a zig-zag chain of
strict equivalences such as $\G\leftarrow \G_1\rightarrow \cdots
\leftarrow\G_n\rightarrow \G'$. Although two strict monoidal abelian
groupoids in the same equivalence class can always be linked by one
intervening pair of strict equivalences, this phenomenon,  we think,
obscures unnecessarily the conclusions. Moreover, the facts stated
in Theorem \ref{concl}$(iv)$ clearly fail for strict monoidal
functors.
\end{remark}

\subsection{Classification of categorical groups revisited }
As we recalled above, a categorical group is a monoidal groupoid
$\G$ in which every object  is invertible or, equivalently, such
that its associated monoid of connected components $M(\G)$ is a
group. By Proposition \ref{proab}, every categorical group is
abelian, so that
$$\mathbf{CatGp}\subseteq \mathbf{MonAbGpd}$$ is the full 2-subcategory
of the 2-category of monoidal abelian groupoids given by the
categorical groups. We shall denote by
$$ \mathbf{Z^3Mnd}|_{\mathbf{Gp}}\subseteq \mathbf{Z^3}\mathbf{Mnd}$$ the full 2-subcategory of the
2-category of Leech $3$-cocycles of monoids whose objects are those
$\mathcal{S}= (G,(\AAw,\Theta),\lambda)$ in $\mathbf{Z^3Mnd}$ where
$G$ is a group. Then,  the biequivalences $(\ref{biequires})$ in
Theorem \ref{thebires} restrict to corresponding quasi-inverse
biequivalences
\begin{equation}\label{bieres}
\xymatrix{\mathbf{CatGp}\ar@{}[r]|-{\thickapprox}
\ar@<4pt>[r]^-{\Delta}&\ar@<3pt>[l]^-{\Sigma} \mathbf{Z^3}\mathbf{Mnd}|_{\mathbf{Gp}}.
}
\end{equation}

But now we shall note that this latter 2-category
$\mathbf{Z^3Mon}|_{\mathbf{Gp}}$  is essentially the same as its
full 2-subcategory, called the {\em $2$-category of  Eilenberg-Mac
Lane $3$-cocycles of groups} \cite{ceg02II} and denoted by
$$
\mathbf{Z^3Gp}\subseteq\mathbf{Z^3Mon}|_{\mathbf{Gp}},
$$
which is defined  by those $\mathcal{S}= (G,(\AAw,\Theta),\lambda)$
as above, but in which the family of groups $\AAw$ is constant, that
is, where $A_a=A_1$ for all $a\in G$, and in the family $\Theta$ all
automorphisms $a^*:A_1\to A_1$, $a\in G$, are identities. Observe
that such a $\Sw$ is then described simply as a triple
$\Sw=(G,(A,\theta),\lambda)$, where $G$ is a group, $A$ ($=A_1$) is
an abelian group, $\theta:G\to\mathrm{Aut}(A)$ is a group
homomorphism ($\theta(a)=a_*$), and $\lambda\in\Z^3(G,(A,\theta))$
is an ordinary normalized 3-cocycle of the group $G$ with
coefficients in the $G$-module $(A,\theta)$, that is, the $G$-module
defined by the abelian group $A$ with left action  $(a,f)\mapsto
{}^af=\theta(a)(f)$.

A morphism
 $(p,q,\varphi):(G,(A,\theta),\lambda)\to (G',(A',\theta'),\lambda')$
in $\mathbf{Z^3Gp}$ then consists of a group homomorphism $p:G\to
G'$, a homomorphism of $G$-modules $$q:(A,\theta)\to
p^*(A',\theta')=(A',\theta'p),$$ and a normalized 2-cochain
$\varphi\in C^2(G,(A',\theta'p))$ such that
$q_*(\lambda)=p^*(\lambda')\circ
\partial^2\varphi$.  If $$(p,q,\varphi), (\bar{p},\bar{q},\bar{\varphi}) :(G,(A,\theta),\lambda)\to
(G',(A',\theta'),\lambda')$$  are two morphisms in $\mathbf{Z^3Gp}$,
then there is no deformation between them unless $p=\bar{p}$ and
$q=\bar{q}$, and, in such a case, a deformation
$\delta:(p,{q},\varphi)\Rightarrow (p,q,\bar{\varphi})$ consists of
a 1-cochain $\delta \in C^1(G,(A',\theta'p))$, such that
$\varphi=\bar{\varphi}\circ
\partial^1\delta$.

 We have a 2-functor
$$(~)_1:\mathbf{Z^3Mon}|_{\mathbf{Gp}}\to \mathbf{Z^3Gp}$$ that is
given on objects by
\begin{equation}\label{tricol}
(G,(\AAw,\Theta),\lambda)\mapsto (G,(A_1,\theta),\widehat{\lambda}),
\end{equation}
where the homomorphism $\theta:G\to\mathrm{Aut}(A_1)$ is defined, by
means of the isomorphisms $A_1\overset{a_*\ }\to A_a \overset{\
a^*}\leftarrow A_1$, $a\in G$, of $\Theta$, by the equations
$$
a^*\theta(a)=a_*,
$$
while the component at any $(a,b,c)\in G\times G\times G$ of the
3-cocycle $\widehat{\lambda}\in \Z^3(G,A_1)$ is defined, by means of
the isomorphism $(abc)^*:A_1\to A_{abc}$, by
$$
(abc)^*(\widehat{\lambda}_{a,b,c})= \lambda_{a,b,c}.
$$

A morphism $(p,\mathbbm{q},\varphi):(G,(\AAw,\Theta),\lambda)\to
(G',(\AAw',\Theta'),\lambda')$ in $\mathbf{Z^3Mon}|_{\mathbf{Gp}}$ is
mapped by the 2-functor $(~)_1$ to the morphism
\begin{equation}\label{morphis}
(p,q_1,\widehat{\varphi}):(G,(A_1,\theta),\widehat{\lambda})\to
(G',(A'_1,\theta'),\widehat{\lambda}'),
\end{equation}
where $\widehat{\varphi}\in C^2(G,(A'_1,\theta'p))$ is the 2-cochain
whose component at any pair $a,b\in G$ is determined by the
isomorphism $p(ab)^*:A'_1\to A'_{p(ab)}$ such that
$$
p(ab)^*(\widehat{\varphi}_{a,b})=\varphi_{a,b},
$$
whereas  a deformation $\delta:(p,\mathbbm{q},\varphi)\Rightarrow
(p,\mathbbm{q},\psi)$ in $\mathbf{Z^3Mon}|_{\mathbf{Gp}}$ is carried
to the deformation in $\mathbf{Z^3Gp}$
$$
\widehat{\delta}:(p,q_1,\widehat{\varphi})\Rightarrow (p,q_1,\widehat{\psi}),
$$
where $\widehat{\delta}\in C^1(G,(A'_1,\theta'p))$ is the 1-cochain
defined by means of the isomorphisms $p(a)^*:A'_1\to A'_{p(a)}$,
$a\in G$, such that
$$
p(a)^*(\widehat{\delta}_a)=\delta_a.
$$

All the needed verifications to prove that $(~)_1$  is actually a
2-functor are quite straightforward. For example, we see that
$\widehat{\lambda}$ in $(\ref{tricol})$ is certainly a 3-cocycle and that
the homomorphism  $q_1:(A_1,\theta)\to (A'_1,p^*\theta')$ in
$(\ref{morphis})$ is of $G$-modules, as follows:
\begin{equation}\nonumber \begin{split}
(abcd)^*\,({}^{a}\!\widehat{\lambda}_{b,c,d}\circ
\widehat{\lambda}_{a,bc,d}\circ &
           \widehat{\lambda}_{a,b,c})\\
&=(bcd)^*a^*({}^{a}\!\widehat{\lambda}_{b,c,d})\circ
(abcd)^*(\widehat{\lambda}_{a,bc,d})\circ
d^*(abc)^*(\widehat{\lambda}_{a,b,c})\\
&=(bcd)^*a_*(\widehat{\lambda}_{b,c,d})\circ \lambda_{a,bc,d}\circ d^*(\lambda_{a,b,c})\\
&= a_*( {\lambda}_{b,c,d})\circ \lambda_{a,bc,d}\circ d^*(\lambda_{a,b,c})\\
&=\lambda_{a,b,cd}\circ \lambda_{ab,c,d}\\
&=(abcd)^*(\widehat{\lambda}_{a,b,cd}\circ \widehat{\lambda}_{ab,c,d}),
\end{split}
\end{equation}
whence ${}^{a}\widehat{\lambda}_{b,c,d}\circ
\widehat{\lambda}_{a,bc,d}\circ
           \widehat{\lambda}_{a,b,c}= \widehat{\lambda}_{a,b,cd}\circ
           \widehat{\lambda}_{ab,c,d}.
$
$$p(a)^*({}^{p(a)}q_1(f))=p(a)_*(q_1(f))\overset{(\ref{yano})}=q_aa_*(f)=q_a(a^*({}^af))
\overset{(\ref{yano2})}=p(a)^*(q_1({}^af)),$$ whence
$q_1({}^af)={}^{p(a)}q_1(f)$.

\begin{proposition} The $2$-functors inclusion  and
$(~)_1$ are mutually quasi-inverse biequivalences
$$
\xymatrix{\mathbf{Z^3Gp}\ar@<-4pt>[r]^(.43){\thickapprox}_-{{in}} &
\mathbf{Z^3}\mathbf{Mnd}|_{\mathbf{Gp}}.\ar@<-5pt>[l]_-{(~)_1}}
$$
\end{proposition}

\begin{proof} We have $(~)_1\,in=1$, the identity,  while the pseudo-equivalence $in\,(~)_1\simeq
1$ is given, at any object $(G,(\AAw,\Theta),\lambda)$,  by the
isomorphism $$
(1_G,\mathbbm{q},1):(G,(A_1,\theta),\widehat{\lambda})\cong(G,(\AAw,\Theta),\lambda)
,$$ where $\mathbbm{q}=(A_1\overset{a^*}\longrightarrow A_a)_{a\in
G}$.\qed
\end{proof}

Therefore, by composing the biequivalences above with those in
$(\ref{bieres})$, we get the following (already known, see
\cite[Theorem 3.3]{ceg02II}) cohomological description of the
2-category of categorical groups:
\begin{theorem} The $2$-functors $\Delta_1=(~)_1\Delta$ and $\Sigma_1=\Sigma\,{in}$,
\begin{equation*}\label{biequiress}
\xymatrix{\mathbf{CatGp}\ar@{}[r]|-{\thickapprox}
\ar@<4pt>[r]^-{\Delta_1}&\ar@<3pt>[l]^-{\Sigma_1} \mathbf{Z^3}\mathbf{Gp}
}
\end{equation*}
 are quasi-inverse biequivalences.
\end{theorem}

Let us now denote by $\mathbf{H^3Gp}\subseteq \mathbf{H^3Mnd}$ the
full subcategory of the category  of Leech $3$-cohomology classes of
monoids $(\ref{h3mon})$, given by the Eilenberg-Mac Lane
$3$-cohomology classes of groups. An object in $\mathbf{H^3Gp}$ is
then a triple $(G,(A,\theta),c)$, where $G$ is a group, $(A,\theta)$
is a $G$-module, and $c\in \HH^3(G,(A,\theta))$. An arrow
$$(p,q):(G,(A,\theta),c)\to (G',(A',\theta'),c')$$ in $\mathbf{H^3Gp}$
consists of a group homomorphism $p:G\to G'$ and a homomorphism of
$G$-modules $q:(A,\theta)\to (A',\theta'p)$ such that
$p^*(c')=q_*(c)\in \HH^3(G,(A',\theta'p))$.

We have the {\em cohomology class functor}
\begin{eqnarray}&\mathrm{cl}: \mathbf{Z^3Gp}\to
\mathbf{H^3Gp}.\nonumber \\[4pt] \nonumber
&(G,(A,\theta),\lambda)\mapsto (G,(A,\theta),[\lambda])\\ \nonumber
&(p,q,\varphi)\mapsto (p,q)\end{eqnarray}
This functor $\mathrm{cl}$ carries
isomorphic morphisms of $\mathbf{Z^3Gp}$ to the same morphism in
$\mathbf{H^3Gp}$ and is surjective on objects. Moreover, it reflects
isomorphisms and is full: if $(p,q,\varphi):(G,(A,\theta),\lambda)\to
(G',(A',\theta'),\lambda')$ is any morphism in $\mathbf{Z^3Gp}$ such
that the maps $p$ and $q$ are invertible, then the morphism of
$\mathbf{Z^3Gp}$
$$(p^{-1}, q^{-1},
{p^*}^{-1}q_*^{-1}(\varphi^{-1})):(G',(A',\theta'),\lambda')\to
(G,(A,\theta),\lambda)$$ is an inverse of $(p,q,\varphi)$. To see that $\mathrm{cl}$ is full,
let $$(p,q):\mathrm{cl}(G,(A,\theta),\lambda)\to
\mathrm{cl}(G',(A',\theta'),\lambda')$$ be any morphism in
$\mathbf{H^3Gp}$, then $p^*[\lambda']$ and $q_*[\lambda]$ both
represent the same class in $\HH^3(G,(A',\theta'p))$, so there is
$\varphi\in C^2(G,(A',\theta'p))$ such that
$q_*(\lambda)=p^*(\lambda')\circ
\partial^2\varphi$. Then,  $(p,q,\varphi):(G,(A,\theta),\lambda)\to  (G',(A',\theta'),\lambda') $
is a morphism in $\mathbf{Z^3Gp}$ with
$\mathrm{cl}(p,q,\varphi)=(p,q)$. Furthermore, let us observe that any other realization
of $(p,q)$ is of the form  $(p,q,\varphi\circ \phi)$ with $\phi\in
\Z^2(G,(A',\theta'p))$ and, moreover, that there is a deformation
$(p,q,\varphi)\Rightarrow (p,q,\varphi\circ \phi)$ if and only if
$\phi= \partial^1\delta$ for some $\delta\in C^1(G,(A',\theta'p))$.

Hence, the {\em classifying functor}
$$\mathrm{Cl}=\mathrm{cl}\,\Delta_1:\mathbf{CatGp}\to \mathbf{H^3Gp}$$
has the following properties:

\begin{theorem}[\cite{Si} Classification of categorical groups]
$(i)$ For any group $G$, any $G$-module $(A,\theta)$, and any
cohomology class $c\in \HH^3(G,(A,\theta))$, there is a categorical
group $\G$ with an isomorphism $\mathrm{Cl}(\G)\cong
(G,(A,\theta),c)$.

$(ii)$ A monoidal functor between categorical groups $F:\G\to\G'$ is
an equivalence if and only if the induced
$\mathrm{Cl}(F):\mathrm{Cl}(\G)\to\mathrm{Cl}(\G')$ is an
isomorphism.

$(iii)$ If $\G$ and $\G'$ are categorical groups, then, for any
isomorphism $(p,q):\mathrm{Cl}(\G)\cong \mathrm{Cl}(\G')$,
 there is a monoidal equivalence
 $F:\G\overset{_\sim\ }\to\G'$
 such that $\mathrm{Cl}(F)=(p,q)$.

$(iv)$ If $\G$ and $\G'$ are categorical groups with
$\mathrm{Cl}(\G)=(G,(A,\theta),c)$ and
$\mathrm{Cl}(\G')=(G',(A',\theta'),c')$, then,  for any morphism
 $(p,q):\mathrm{Cl}(\G)\to \mathrm{Cl}(\G')$
 in  $\mathbf{H^3Gp}$, there is a
(non-natural) bijection
$$
\big\{[F]:\G\to\G' \mid \mathrm{Cl}(F)=(p,q)\big\}\cong \HH^2(G,(A',\theta'p)),
$$
between the set of isomorphism classes of those monoidal functors
$F:\G\to \G'$ that are carried by the classifying functor to $(p,q)$
and  the second cohomology group of $G$ with coefficients in the
$G$-module $(A',\theta'p)$.
\end{theorem}

\end{document}